\documentclass[12pt, reqno]{amsart}
\usepackage[utf8]{inputenc}
\usepackage[english]{babel}

\usepackage{mathrsfs}
\usepackage{amsthm}
\usepackage{amsmath}
\usepackage{amsfonts}
\usepackage{mathtools}
\usepackage{amssymb}

\usepackage{textcomp}
\usepackage{xcolor}

\usepackage{calrsfs}
\usepackage{csquotes}

\theoremstyle{plain}
\newtheorem {theorem}{Theorem}[section]
\newtheorem {lemma}[theorem]{Lemma}
\newtheorem {corollary} [theorem]{Corollary}

\newtheorem {proposition} [theorem]{Proposition}
\theoremstyle{definition}
\newtheorem{definition}[theorem]{Definition}
\newtheorem{remark}[theorem]{Remark}

\theoremstyle{remark}

 \numberwithin{equation}{section} 
 
\newcommand{\C}{\mathbb{C}}

\newcommand{\R}{\mathbb{R}}

\newcommand{\N}{\mathbb{N}}

\newcommand{\D}{\mathcal{D}}
\renewcommand{\L}{\mathcal{L}}
\newcommand{\I}{\mathcal{I}}
\newcommand{\A}{\mathcal{A}}
\newcommand{\B}{\mathcal{B}}
\newcommand{\T}{\mathcal{T}}

\newcommand{\oo}{\infty}

\newcommand{\eee}{\mathrm{e}}

\renewcommand{\phi}{\varphi}

\newcommand{\s}{\sigma}
\newcommand{\0}{\theta}

\renewcommand{\r}{\rho}
\renewcommand{\d}{\partial}
\renewcommand{\a}{\alpha}
\renewcommand{\b}{\beta}
\renewcommand{\k}{\kappa}
\newcommand{\g}{\gamma}
\newcommand{\la}{\lambda}

\newcommand{\De}{\Delta}

\newcommand{\z}{\zeta}

\newcommand{\1}{{-1}}
\newcommand{\q}{\bar q}
\newcommand{\Ga}{\Gamma}
 \newcommand{\m}{n}

\newcommand{\sgn}{\mathrm{sgn}}

\renewcommand{\gg}{\dot{\gamma}}

\newcommand{\ad}{\mathrm{ad}}

\renewcommand{\.}{\dots}

\newcommand{\Ad}{\mathrm{Ad}}
\newcommand{\eexp}{\overrightarrow{\exp}}

\newcommand{\mc}{\mathcal}
\newcommand{\mr}{\mathrm}

\newcommand{\IM}{\mathrm{Im}}
\newcommand{\proj}{\mathrm{proj}}

\newcommand{\coker}{\mathrm{coker}}

\newcommand{\coim}{\mathrm{coim}}
\newcommand{\dom}{\mathrm{dom}}

\newcommand{\be}{\begin{equation}}
\newcommand{\ee}{\end{equation}}

\renewcommand{\A}{\mc A}
\renewcommand{\z}{\mathrm z}
\renewcommand{\k}{\mathrm k}
\newcommand{\J}{\mathcal J}
\renewcommand{\Re}{\mathrm {Re}}

 \newcommand{\ddue}{\frac{\d}{\d x_2}}
 \newcommand{\dtre}{\frac{\d}{\d x_3}}

\newcommand{\chia}{\chi_{\left[0,\frac{1}{2}\right]}}
\newcommand{\chib}{\chi_{\left[\frac{1}{2},\frac{3}{4}\right]}}
\newcommand{\chic}{\chi_{\left[\frac{3}{4},1\right]}}

\title[Goh conditions]{Higher order Goh conditions  for  singular extremals of corank $1$}
 
\author[F.~Boarotto]{Francesco Boarotto$^1$} 

\author[R.~Monti]{Roberto Monti$^2$} 

\author[A.~Socionovo]{Alessandro Socionovo$^3$}

\email{francesco.boarotto@gmail.com$^1$, monti@math.unipd.it$^2$, alesocio1994@gmail.com$^3$}

\address
{$^{1,2,3}$Universit\`a di Padova, Dipartimento di Matematica  ``Tullio Levi-Civita'',
	via Trieste 63, 35121 Padova, Italy}

\address{$^3$Laboratoire Jacques-Louis Lions, CNRS, Inria, Sorbonne Université, Université de Paris, France}


\begin{document}

\begin{abstract}  We prove Goh conditions of order $n\geq 3$ for strictly singular length-minimizing curves of corank $1$, under the assumption  that   
{\color{black}
the domain of the $n$th instrinsic differential is of finite codimension. }This result relies upon the proof of an open mapping theorem for maps with  regular $n$th differential.
\end{abstract}

\maketitle

  \newenvironment{dedication}
        {\vspace{1ex}\begin{quotation}\begin{center}\begin{em}}
        {\par\end{em}\end{center}\end{quotation}}

        {\color{black}
        
   \begin{dedication}
\vspace*{0cm}{Dedicated to Andrei Agrachev on his 70th birthday}
\end{dedication}

}
 
\section{Introduction}\label{sec:zero}

\renewcommand{\proj}{\mathrm{pr}}

One of the main open problems in sub-Riemannian geometry is the regularity of length-minimizing curves. 
Its difficulty is due to the singularities of the end-point map, i.e., to the presence of points where its differential is not surjective.
In this paper, we study the end-point map up to order $n$ obtaining necessary conditions of Goh-type for optimal trajectories. 
The fine understanding of sub-Riemannian geodesics with their best regularity is of great importance in several fields ranging from {\color{black} Nonholonomic  Mechanics to Geometric Control Theory.}

A sub-Riemannian manifold is a triplet $(M,\De, g)$, where $M$ is a smooth, i.e., $C^\infty$ manifold, $\De\subset TM$ is a distribution of rank $2\leq d<\dim(M)$, and $g$ is a metric on $\De$. In a neighborhood $U\subset M$ of any point  $q\in M$, there exist vector-fields  $f_1,\.,f_d\in \mr{Vec}(U)$   such that $\De=\mr{span}\{f_1,\.,f_d\}$ on $U$. Since our considerations are local, we can assume in the sequel that $U=M$.
We also assume that $\De$ satisfies  H\"ormander's condition  
\begin{equation}\label{eq:horm}
 		\mr{Lie}\{f_1,\.,f_d\}(p)=T_pM, \quad p\in M,
\end{equation}
i.e.,  that $\De$ is completely non-integrable.  For an exhaustive introduction to sub-Riemannian geometry, we refer the reader to \cite{ABB20, AgrSac, B96, J14, Mon02, Rif14}.

Let $I=[0,1] $ be the unit interval. A curve $\g\in AC(I;M)$ is  \emph{horizontal} if $\gg\in \De_{\g}$ a.e.~on $I$, that is
\be
	\label{curvaor}
 	\gg(t)=\sum_{i=1}^d u_i(t)f_i(\g(t)), \quad \text{for a.e.~$t\in I$} ,
\ee
for some unique $u=(u_1,\.,u_d)\in L^1(I;\R^d)$,  called   control of $\g$. 
Without loss of generality, we can   assume that 
  $g$  makes $f_1,\.,f_d$   orthonormal, in which case   
the length of $\g$ is  the $L^1$-norm of its control. We can also replace the Banach space $L^1(I;\R^d)$ with the smaller Hilbert space $X= L^2(I;\R^d)$.

The end-point map $F_{q}:X\to M$ with base-point $q\in M$ is defined letting  $F_q(u)=\g_u(1)$, where $\g_u$ is the unique solution to \eqref{curvaor} with $\g_u(0)=q$. The point $\q =F_q(u)$ is the end-point of the curve $\g_u$. 
 Since $q\in M$ is fixed, we shall simply write $F=F_{q}$.

 {\color{black} Controls $u\in X$ where the differential $d_uF$ is not 
 surjective are called \emph{singular}. Now consider the extended end-point   map $F_J: X\to M\times\R$,
	 $F _J(u) = (F(u), \frac12  \| u\|_2^2 )$. 
If  for every $(\lambda,\lambda_0)\in 
	\IM(d_u F_J  )^\perp$ we have  $\lambda_0=0$, the singular control $u$ is called 
\emph{strictly singular}, see Definition \ref{defi:stricts}.
If $u$ is not strictly singular 	(namely, if it is \emph{normal}) then a length-minimizing
curve  $\gamma_u$ is smooth.
For this reason the regularity problem of sub-Riemannian geodesics reduces to the regularity of strictly singular minimizers.}

In his ground-breaking work \cite{Mon94}, 
 Montgomery  first proved that {\color{black} strictly singular} curves can be as a matter of fact  length-minimizing.
 His example was discovered studying a
 charged particle traveling in the plane under
the influence of a magnetic field.
 Also   \emph{nice abnormal extremals}, see  \cite{LS95}, are locally length-minimizing.
 Examples of purely Lipschitz and spiral-like abnormal curves in Carnot groups are presented in \cite{LDLMV13,LDLMV18}, and an algorithm for producing many new examples is proposed in \cite{H20}.  The length-minimality property of all these examples is not  yet well-understood.

 A   recent approach  to the regularity problem of length-minimizing curves is based on the analysis of specific singularities such as corners, spiral-like curves or
 curves with no  straight tangent line. This approach  does not use   open mapping theorems but it rather relies on the ad hoc construction of shorter competitors, see 
  \cite{BCJPS20, HL16, HLprep, LM08, Monti14, MPV18, MS21}.
  
  {\color{black} Another new and interesting approach to the problem is proposed in \cite{Lev22}, where the authors prove that the controls of strictly singular length-minimizers are $L^p$-H\"{o}lder continuous.}

On the other hand, necessary conditions for the minimality of singular extremals can be obtained from the differential study of the end-point map. The   theory  is well-known till the second order and was initiated by Goh \cite{Goh66} and developed by Agrachev and Sachkov in \cite{AgrSac}. Using second order open mapping theorems (index theory), 
for a strictly singular length-minimizing curve $\gamma$ and for any adjoint curve $\lambda$ they  prove the validity of the following Goh conditions:
\be\label{SE}
\langle\lambda, [f_i,f_j](\gamma)\rangle =0,  \quad \textrm{ $i,j=1,\ldots,d$.}
\ee
 The first order conditions $\langle\lambda, f_i (\gamma)\rangle =0$ are ensured by 
 Pontryagin Maximum Principle. Partial necessary conditions  of the third order  are obtained  in \cite{BMP20}. {\color{black} 
 Generalized second order Goh conditions have been recently obtained in \cite{Joz21}.}

Our goal is to extend the second order theory of \cite{AS96} to any order $n\geq 3$ and to get necessary conditions as in \eqref{SE} involving brackets of $n$ vector fields.

There is a clear connection between the geometry of $\Delta$ and the expansion of the end-point map $F$. In particular, the commutators of length $n$ should appear  in the $n$th order term of the expansion of $F$. In Section \ref{EXPA}, we provide a first positive answer to this idea.

In order to develop the theory,  we need a suitable definition of $n$th   differential.
For $v_1,\ldots,v_n \in X$ and  $u\in X$, we first define
\[
D_u^n F(v_1,\dots,v_n) =
\frac{d^n}{dt^n} F\Big( u+\sum_{i=1}^n\frac{t^i v_i}{i!}\Big)\bigg|_{t=0} .
\]
Then, we restrict $D_u^n F$  to a suitable domain  $\dom(\mc D_u^n F)\subset X^{n-1}$ that, roughly speaking, consists of points where the lower order differentials
$d_uF, D_u^2F,\dots, D_u^{n-1}F$ vanish. Finally,   we define 
$\mc{ D}_u^n F :\dom(\mc D_u^n F)\to \coker(d_u F)$ letting
$\mc{ D}_u^n F = \proj(
D_u^n F)$, where   $\proj$   is the  projection onto $ \coker(d_u F)$, see Definition~\ref{definition:diffs}.
A motivation for this definition is the fact that 
$\mc{ D}_u^n F $ behaves covariantly, in the sense that, for a given  diffeomorphism $P\in C^\infty(M;M)$, 
$\mc{ D}_u^n (P\circ F)$ depends only on the first order derivatives of $P$.

{\color{black}  
If the set of all $v_1\in \ker(d_uF)$ that can be extended to some $v=(v_1,\ldots,v_{n-1})\in \dom (\mc D_u^n F)$
contains a linear space of finite codimension in $X$, we say that $\dom (\mc D_u^n F)$ has finite codimension, see Definition \ref{FCD}.
This property is in general difficult to check. However, 
if the lower intrinsic differentials vanish
	\be\label{eq:Gohn}
		\mc{D}^h_u F =0, \ \  h=2,\dots,n-1,
	\ee
then $\dom (\mc D_u^n F)$ automatically has finite codimension. This is a corollary of Proposition \ref{proposition:fincodim}.
}

Our main result consists of necessary conditions of Goh-type for length-minimizing strictly singular curves 
 $\g_u$  of corank-one, i.e., such that  $\mathrm{im}(d_uF)$ has codimension  $1$ in $T_{\g_u(1)}M$.
For the definition of adjoint curve, see Section~\ref{GOH}

\begin{theorem}\label{Gohnintro}
	Let $(M,\Delta,g)$ be a sub-Riemannian manifold and  $\gamma=\gamma_u \in AC ( I ; M)$ be a strictly singular length-minimizing curve of corank $1$.
	{\color{black} If  $\dom(\mc D_u ^nF)$,  $n\geq 3$,  has finite codimension} then any adjoint curve $\la \in AC( I ;  T^*M)$ satisfies   \begin{equation}
 \label{DELTAintro}
 \langle \la(t), [f_{j_n},[\.[f_{j_{2}},f_{j_1}]\.]](\gamma(t)) \rangle=0,
 \end{equation} 
 for all  $t\in I $ and for all  $j_1,\.,j_n=1,\.,d$.
 \end{theorem}

 {\color{black} In general, necessary conditions as in \eqref{SE} and \eqref{DELTAintro} are not enough to prove the non-minimality of corners or spirals  (see for instance the example on page 17 in \cite{Joz21}).}

The proof of Theorem~\ref{Gohnintro} relies on an open mapping argument applied to the extended end-point map  $F_J =(F,J) : X\to M\times \R$, where $J(u) = \frac 12 \| u\|^2_{L^2(I;\R^d)}$ is the energy of $\gamma=\gamma_u$. Minimizing the energy is in fact equivalent to minimizing the length, because for horizontal curves parameterized by arc-length the $L^2$-norm of the control coincides with its $L^1$-norm.

Motivated by this application, in Section~\ref{sectwo}  we develop a  theory about open mapping theorems of order $n$ for functions $F:X\to \R^m$ between Banach spaces. In our opinion, this preliminary study is  worth of interest on its own. It adapts in a geometrical perspective some ideas presented in \cite{Sus03}.

\begin{theorem}\label{thmopen}
 Let $X$ be a Banach space and let $ F\in C^\infty(X;\R^m)$, $m\in\N$,  be a smooth mapping. {\color{black}
 If  the  intrinsic differential $\mc D_0^n F:\dom(\mc D_0^n F)\to\coker(d_0F)$, $n\geq 2$, 
is regular at the critical point $0\in X$ then $F$ is open at $0$.}
\end{theorem}

{\color{black}
The notion of \enquote{regularity} used in Theorem \ref{thmopen} is delicate because   $\mc D_0^n F $ 
is a non-linear mapping defined on a domain without linear structure. Denoting by $\ell\in\{1,\ldots,m\} $ the corank of the critical point $0$, 
Definition 
\ref{REGO} of regular differential depends on 
the existence of some homogeneous map from $\R^\ell$ into $\dom(\mc D_0^n F)$ inverting $\mc D_0^n F$ in a bounded way.
Under the assumption of vanishing lower differentials  \eqref{eq:Gohn} this homogeneous map can be constructed  explicitly, see Proposition \ref{prop:corrrections}.
When $0\in X$ is a critical point of corank $1$ the definition of regular differential is effective, see Corollary \ref{REGONO}.
}

The rest of the paper is devoted to the study of the open mapping property for the extended end-point map $F_J$ around a singular control $u$. 
In fact, we will study the auxiliary map $G$, called  \emph{variation map}, 
defined by $G(v) = F(u+v)$,
in order to move the base-point from $u$ to $0$.

A crucial ingredient in our analysis is the definition of non-linear  sets $V_{h}\subset X$, $h\in \N$,
consisting of   controls with vanishing  iterated integrals for any order $h\leq n-1$,  see \eqref{bip}. Using such controls we are able to catch the geometric structure of the $n$th differential $\mc D_0^nG$ in terms of Lie brackets. {\color{black} 
The algebraic properties of the sets $V_h$ appear in  the theory of rough paths (see for instance \cite{Lyo}) and
are studied in 
Section \ref{secthree}}.

In {\color{black}Sections \ref{EXPA} and \ref{OMP},}  we use the formalism of chronological calculus \cite[Chapter 2]{AgrSac} to compute the $n$th   differential $\mc{D}_0^n G$ of the variation map
and the final outcome is formula \eqref{pluto5}. This formula contains a localization parameter $s>0$ that can be used to shrink the support of the control in a neighborhood of some point $t_0\in [0,1)$. Passing to the limit as $s\to0^+$, we obtain  a new map 
 $\mc G^n _{t_0}: X\to \R$:
\[
	\mc G^n _{t_0}(v) = \int_{\Sigma_n } \langle \la,[ g_{v(t_n)}^{t_0} , [ \.,[g_{v(t_2)}^{t_0},  g_{v(t_1)}^{t_0}]]\ldots] (\q ) \rangle dt_1\dots dt_n,
\]
where $\Sigma_n =\{ 0\leq  t_n\leq t_{n-1}\leq\ldots\leq t_1 \leq 1\}$ is the standard simplex, 
$\lambda \in T^*_{\q }M$ is a fixed covector orthogonal to $\coker(d_u F)$, 
$\q  = F(q)\in M $ is the end-point,
and 
$
g_{v(t_i)}^{t_0} $ is the pull-back of the time-dependent vector field $f_v= v^1 f_1 +\ldots + v^d f_d$ along the flow of $u$. 
In the corank $1$ case, we show that  if there exists $v\in V_{n-1}$ with 
	$\mc G^n _{t_0}(v) \neq0$
then the extended map $G_J$ is open at $0$, see  our  Theorems \ref{endopen} and 
\ref{endopen2} where the hypothesis on $\dom(\mc D _0^nG)$ to have finite codimension is crucial. So $\mc G^n _{t_0}=0 $ on $V_{n-1}$ becomes a necessary condition for the length-minimality of singular extremals.

In Sections~\ref{GJI} and \ref{NSVTF}, we study the geometric implications of   equation $\mc G^n _{t_0}=0$. 
First, we explore the symmetries of  $\mc G^n _{t_0}$, showing how the shuffle algebra of   iterated integrals  interacts with  generalized Jacobi identities of order $n$, see Theorem \ref{lisk1}. In spite of the non-linear structure of $V_{n-1}$, we are able to polarize the equation $\mc G^n _{t_0}=0$ on linear subspaces of $V_{n-1}$ 
of arbitrarily large dimension, thus  de facto bypassing the non-linearity of the problem. 

At this point,  we regard
the quantities in \eqref{DELTAintro} as    unknowns  of a nonsingular system of linear equations, thus proving their vanishing.  
To get this nonsingularity, we   work with families of trigonometric functions having sparse and high frequences, see Theorems \ref{TRIX} and  \ref{bugsbunny}.

Our   argument leading to the final proof of Theorem \ref{Gohnintro} is summarized in  Section~\ref{GOH}.
In Section 10 we complete the study of a well-known example of {\color{black} a strictly  singular} curve.  
In $M =\R^3$  we consider the distribution spanned by the vector-fields
\begin{align}
\label{eq:vfieldsintro}
f_1=\frac{\partial}{\partial x_1} \quad \textrm{and}\quad
f_2=(1-x_1)\frac{\partial}{\partial x_2}+x_1^\m \frac{\partial}{\partial x_3},
\end{align}
where   $\m\in \N$ is a parameter. Using the theory developed in Sections 2-6,  we 
 show that for odd $\m$ the curve $\gamma(t) = (0,t,0)$, $t\in[0,1]$,  is not length-minimizing.
This is   interesting   because, for even $\m$, this singular curve is on the contrary   locally length-minimizing.
In fact, for even $n$ the $n$th differential of the end-point map  does not satisfy assumption i) of Proposition  \ref{REGON} and our open mapping theorem does not apply. See our discussion in Remark \ref{R104}.

We conclude this introductory part commenting on the assumptions made in Theorem \ref{Gohnintro}. {\color{black}
The assumption on $\dom(\mc D_u^n F)$ to have finite codimension cannot be easily dropped:
it is used in the key   Theorem \ref{endopen}.}  The corank $1$ assumption on the length-minimizing curve  is used when Theorem \ref{thmopen}
is applied to the end-point map. We think that it should be possible to drop the corank $1$ assumption, 
but this certainly requires some new deep idea.

\medskip

{
\bf
Acknowledgements. \rm We thank Marco Fantin for illuminating discussions about open mapping theorems.
}

\medskip

{\bf Research Funding. \rm  This project has received funding from the European Union’s Horizon 2020 research and innovation programme under the Marie Sk{\l}odowska-Curie grant agreement No 101034255. \includegraphics[width=0.65cm]{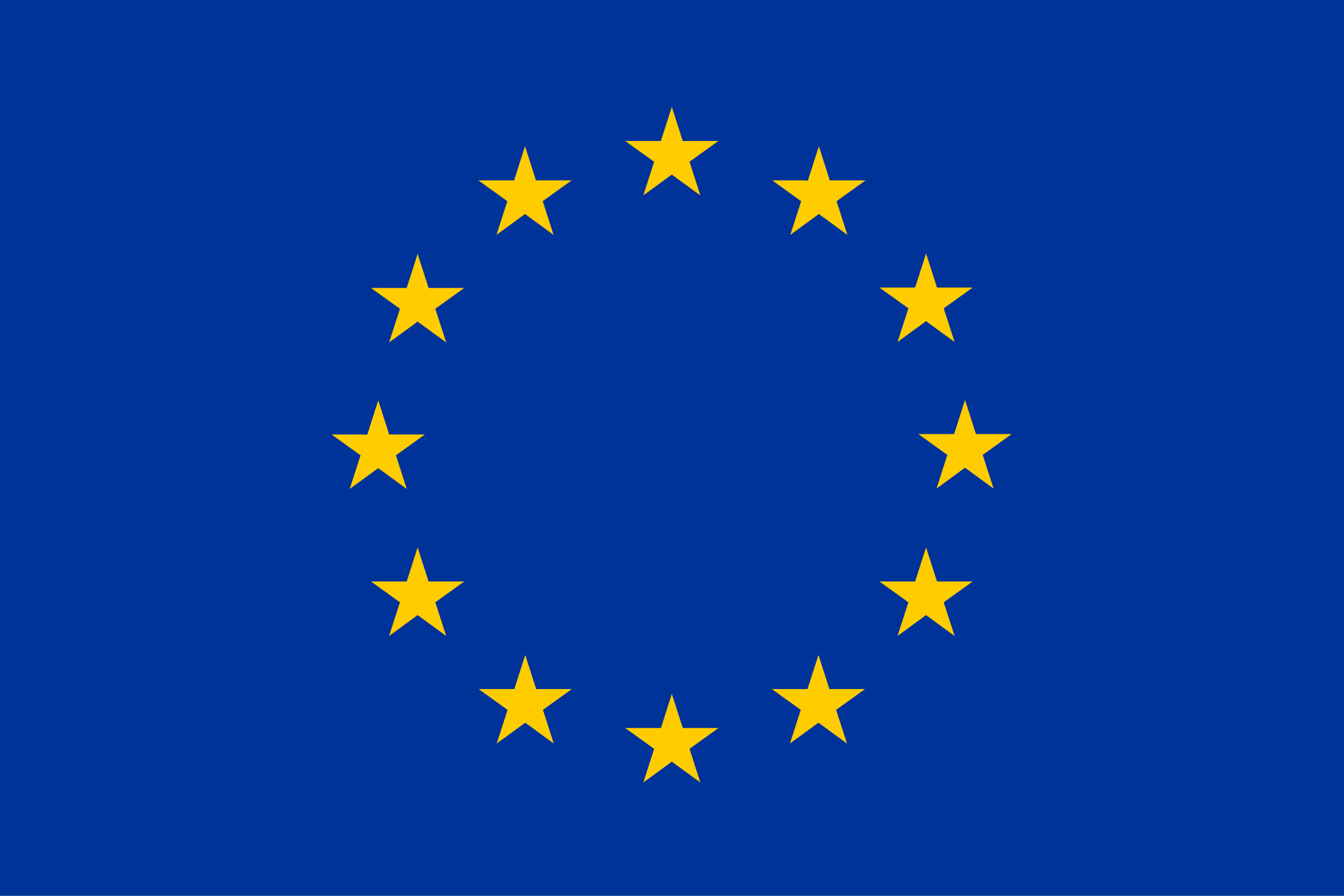}}

\section{Intrinsic differentials}\label{sec:one}

Let $X$ be a Banach space and $F\in C^\infty (X; \R^m)$, $m\in\N$,  be a smooth  map. 
For any   $n\in \N$ we define the $n$th differential of $F$ at $0\in X$ as the map $d_0^n F : X\to\R^m$
 \[
 d_0^nF (v)=\frac{d^n}{dt^n} F(tv)\big|_{t=0},\quad v\in X.
 \]
 With abuse of notation, the associated $n$th multilinear differential is the map  $d _0^n F : X^n \to\R^m$ defined in one of the following two equivalent ways,
 for $v_1,\dots,v_n \in X$, 
 \begin{equation}\label{buio}
 \begin{split}
 d_0^nF (v_1,\dots,v_n)=& \frac{\partial ^n}{\partial t_1 \dots \partial t_n} 
 F\Big(\sum_{h=1}^n t_h v_h\Big)\bigg| _{t_1=\dots=t_n=0}
 \\
 =&\frac{1}{n!} \frac{\partial ^n}{\partial t_1 \dots \partial t_n} 
 d_0^n F\Big(\sum_{h=1}^n t_h v_h\Big)\bigg| _{t_1=\dots=t_n=0}
  .
 \end{split}
 \end{equation}
We have the identity $d_0^n F(v) = d_0^n F(v,\dots,v)$. The differential $d_0^nF$ is 
symmetric, in the sense that $d_0^nF (v_1,\dots,v_n)=d_0^nF (v_{\sigma_1},\dots,v_{\sigma_n} )$
 for any permutation $\sigma\in S_n$. Here and hereafter, we use the notation $\s_i= \s(i)$ for a permutation $\s$
  and for $i=1,\ldots,n$.

A different   $n$th multilinear differential for $F$ at $0$ is  the map $D_0^n F:X^n \to \R^m$ defined by the formula
\begin{equation}
\label{DOF}
D_0^n F(v_1,\dots,v_n) =
\frac{d^n}{dt^n} F\Big(\sum_{h=1}^n\frac{t^h v_h}{h!}\Big)\bigg|_{t=0}, \quad v_1,\dots,v_n \in X.
\end{equation}
The multilinear differential $D_0^n F$ is not symmetric.

The $n$th multilinear differentials $d_0^nF $ and $  D_0^nF $ are related via the  Fa\`a di Bruno formula \cite{Joh02}.
We denote by   $\mathscr{I}_n$ the set of $n$-multi-indices, i.e.,
 \[
 \mathscr{I}_n=\left\{\a\mid \a=(\a_1,\dots,\a_n)\in \N^n\right\},
 \]
 where $\N =\{1,2,\dots\}$ starts from $1$. When the naturals start from $0$ we use the notation $ \mathscr{I}_n^0$. Also, for $d\in \N$ we use the notation $\mathscr{I}_{n,d}$ for the sets of $n$-multi-indices $\a=(\a_1,\dots,\a_n)\in \mathscr I_n$ with values in $\{1,\dots,d\}^n$. For  $\a\in \mathscr I_n$, we use the standard notation  $|\a|=\a_1+\dots+\a_n$  for the length (or weight) of $\alpha$ and  $\a!=\a_1!\dots\a_n!$ for its factorial.

%


\begin{proposition}[Fa\`a di Bruno]
%
 For any $n\in\N$ and $v_1,\dots,v_n\in X$ we have 
 \begin{equation}\label{FdB}
 D_0^n F(v_1,\dots, v_n)=\sum_{h=1}^n\sum_{\a\in \mathscr I_h, |\a|=n}\frac{n!}{h!\a!}d_0^hF(v_\a),
 \end{equation}
 where, for $\a\in \mathscr I_h$, we set $v_\a=(v_{\a_1},\dots, v_{\a_h})\in X^h$.
\end{proposition}

The $n$th differential $D^n_0F$, $n\geq 2$,  does not transform covariantly, in the sense that, for a generic diffeomorphism $P\in\ C^\infty(\R^m;\R^m)$ the differential $D^n_0(P\circ F)$ depends also on the derivatives of $P$ of order $2$ and higher. In order to have an ``intrinsic'' $n$th differential, we need to restrict $D_0^nF$ to a suitable domain and project it onto $\coker(d_0F)$.
We denote  by $\proj:\R^m\to \coker(d_0F)$ the standard projection ({\color{black} i.e., the quotient map modulo $\ker(d_0F)$}).
We fix coordinates in $X$ and $\R^m$ in such a way that $\coim(d_0F) =\R^{m-\ell}$ and $\coker (d_0F) = \R^\ell$.
So we have
 the   splittings
\begin{equation}
\label{SPLIT}
X = \ker(d_0 F) \oplus \R^{m-\ell},\qquad
\R^m = \R^{\ell}\oplus \mathrm{im}(d_0 F), 
\end{equation}
where $\ell=\mathrm{dim}(\coker(d_0 F))$. The differential  $d_0 F:  \R^{m-\ell}\to \mathrm{im}(d_0 F)$ is a linear isomorphism.

\begin{definition}[Intrinsic $n$th differential]\label{definition:diffs} 
Let $F\in C^\infty(X;\R^m)$.  By induction on $n\geq 2$,
we define a domain $   \dom(\mc D_0^nF)\subset X^{n-1} $ and a map $\mc D_0^n F:\dom(\mc D_0^n F)\to \coker(d_0F)$, called
 \emph{intrinsic $n$th differential} of $F$ at $0$, in the following way.

 When $n=2$ we let $ \dom(\mc D_0^2F)=\left\{v\in X\mid D_0F(v)=0 \right\}=\ker(d_0F) \subset X$ and we define
$
 \mc D_0^2F:\dom(\mc D_0^2F)\to \coker(d_0F)$  
  \begin{equation}\label{d2}
 \mc D_0^2F(v) =   \proj(D_0^2F(v,*)),\quad v\in \dom(\mc D_0^2F).  
 \end{equation}
 Inductively, for $n>2$ we set  
 \[
  \dom(\mc D_0^nF)=\left\{v \in \dom(\mc D_0^{n-1}F)\times \mathrm{coim}(d_0F) \mid D_0^{n-1}F(v)=0 \right\}\subset X^{n-1},
 \]
 and we define  $  \mc D_0^n F:\dom(\mc D_0^n F)\to \coker(d_0F)$ as 
 \begin{equation}\label{dn}
   \mc D_0^n F(v)=\proj(D_0^nF(v,*)), \quad v \in \dom(\mc D_0^nF).  \end{equation}
\end{definition}

\begin{remark}
 The definition of $\mc D_0^n F$ in \eqref{d2} and \eqref{dn}  does not depend on the last argument $*\in X$  of $D_0^nF$.
 Indeed, by formula \eqref{FdB} with $v=(v_1,\dots,v_n)$, so with  $*=v_n$ in the notation above, we have
 \begin{equation}\label{FUF}
 D_0^nF(v)=d_0 F(v_n) + \sum_{h=2}^n\sum_{\a\in \mathscr I_h, |\a|=n}\frac{n!}{h!\a!}d_0^hF(v_\a),
 \end{equation}
 where $v_\a$ does not contain $v_n$ when $|\a| = n$ and $ h \geq 2$, and $\proj(d_0 F(v_n))=0$.
 \end{remark}

{\color{black}

\begin{remark}  For any $n\in\N$,  we   introduce on  $X^{n}$ the  dilations $\delta_\lambda: X^{n}\to X^{n}$, $\lambda \in \R$,
\[
\delta_\lambda (v_1,v_2,\ldots,v_{n}) = (\lambda v_1,\lambda^ 2 v_2,\ldots,\lambda^{n} v_{n}).
\] 
From formulas \eqref{buio} and \eqref{FdB} it follows that for any $v\in X^n$ we have the $\delta$-homogeneity property
\[
 D_0 ^n F (\delta_\lambda (v)) = \lambda ^n D_0 ^n F(v).
\]
Then the domain  $\dom(\mc D_0^n F)$ is also $\delta$-homogeneous as a subset of $X^{n-1}$:
\be\label{DH}
\delta_\lambda( \dom(\mc D_0^n F)) = \dom(\mc D_0^n F) \quad \textrm{for any nonzero }\lambda \in \R.
\ee
 \end{remark}
}

While  $\dom(\mc D_0^2F)=\ker(d_0F)$ has  finite codimension in $X$, this might not be the case when     $n>2$.  
In order to develop the theory within a consistent setting we need  some additional assumption on $F$.

{\color{black}
\begin{definition}[Domain with finite codimension] 
\label{FCD}
We say that $\dom(\mc D_0^n F)$, $n\geq 3$, has finite codimension if the set 
$\{ v_1 \in \ker (d_0 F )\mid 
\textrm{there exists  }v= (v_1,\ldots, v_{n-1}) \in \dom(\mc D_0^n F) \}$
contains a linear space of finite codimension in $X$. We define the codimension of $\dom(\mc D_0^n F)$ as  the smallest of these codimensions.
\end{definition}

The vanishing of lower order differentials \eqref{eq:Gohn} ensures that $\dom(\mc D_0^n F)$ has finite codimension.
This is a corollary of the following proposition.

}

\begin{proposition}\label{proposition:fincodim} Let 
$F\in C^\infty(X;\R^m)$ be a smooth map.
If $ \mc D_0^h F = 0$ for all $ 2\le h<n$, with $n\geq 3$, then 
 $\dom(\mc D_0^n F)\subset X^{n-1} $ is a nonempty affine bundle over $\ker(d_0F)$ that is diffeomorphic to 
 $\ker(d_0F)^{n-1}$.  
\end{proposition}

\begin{proof} The proof is by induction on $n\geq 3$. When $n=3$  the domain of $\mc D_0^3 F$
is
 \[
  \dom(\mc D_0^3 F)=\big\{(v_1,v_2)\in \ker(d_0F)\times \mathrm{coim}(d_0 F)  \mid D_0^2F(v_1,v_2)=0\big\},
 \]
 where, as in \eqref{FUF}, 
$
 D_0^2 F(v_1,v_2) = d_0 F(v_2) +d ^2_0F(v_1,v_1) .
$

 We use the splittings \eqref{SPLIT}.
 Since the map $d_0 F : \R^{m-\ell} \to \mathrm{im} (d_0 F)$ is invertible, we can define 
  $\phi \in C^\infty( \ker (d_0F ) ,  \R^{m-\ell}) $ letting
 \[
   \phi ( v_1) =-   (d_0 F)^{-1} (d^2_0F(v_1,v_1)).
 \]
 This is well-defined
 because, 
 by assumption, we have   $ \mc D_0^2 F = 0$ and this implies $\proj(d^2_0F(v_1,v_1))=0$.
 Now, letting 
  \[
  \Phi(v_1,v_2) = ( v_1, v_2 +\phi(v_1)),
  \]
we obtain a diffeomorphism  $\Phi:\ker(d_0 F)^2 \to   \dom(\mc D_0^3 F)$.

 Suppose the theorem is true for $n$ and let us prove it for $n+1$. 
Our inductive assumption is the existence of a diffeomorphism  $\Phi \in C^\infty ( \ker (d_0F )^{n-1} ,   \dom(\mc D_0^n F))$.
Now we have  
 \[
 \begin{split}
\dom(\mc D_0^{n+1} F) & =\left\{(v,w)\in \dom(\mc D_0^n F)\times \mathrm{coim}(d_0F) \mid D_0^nF(v,w)=0\right\}
\\
& =\left\{(\Phi(u) ,w)  \mid u\in \ker (d_0F )^{n-1},\, w \in\mathrm{coim}(d_0F),\,
D_0^nF(\Phi(u) ,w)=0\right\},
\end{split}
 \]
 and, by \eqref{FUF},
 equation $D_0^nF(\Phi(u) ,w)=0$ reads
 \[
w=\psi(u) = - ( d_0 F)^{-1}  \sum_{h=2}^n\sum_{\a\in \mathscr I_h, |\a|=n}\frac{n!}{h!\a!}d_0^hF(\Phi(u)_\a).
 \]
 The function $\psi$ is well-defined because $\mc D_0^n F=0$.
 Now $\Psi(u,z) = ( u, z+ \psi(u))$ is a diffeomorphism from $\ker (d_0F )^{n}$ to $ \dom(\mc D_0^{n+1} F)$.
 
\end{proof}

{\color{black}
Our next goal is to introduce the notion of regular $n$th differential. Recall that 
  $0\in X$ is a critical point of $F$ with corank $\ell\in\{1,\dots,m\}$ if  
 $\dim( \coker (d_0F)) =  \ell$.
 
 \begin{definition} Let $\ell\geq 1$ and $n\geq 2$. 
 We say that a continuous map $w=(w_1,\ldots, w_{n-1}): \R^\ell\to X^{n-1}$ is separately $\delta$-homogeneous if there exist vectors 
 $v_j^{\b,\alpha} \in X$, for  $j\in\{1,\ldots,n-1\}$, $\b\in  \mathscr I_\ell^0$ with $1\leq |\beta|\leq n-1$, and $\alpha \in \{-1,1\}^\ell$, such that  for all $t\in\R^\ell$ we have 
 \begin{equation} 
 \label{wallo}
 w_j(t)=\sum_{\b\in \mathscr I_\ell^0,|\b|=j}t^\b v_j^{\b,\mathrm{sgn} (t)} ,\quad j=1,\ldots, n-1,
\end{equation}
where $\mathrm{sgn} (t )= ( \mathrm{sgn}( t_1 ),\ldots,\mathrm{sgn} (t_1) )$. 

\end{definition}

For constant $\mathrm{sgn}(t)$, the coordinate $w_j$ is  a homogeneous polynomial in $t\in\R^\ell$ with degree $j$ and coefficients in $X$.
If in \eqref{wallo} the coefficients $v_j^{\b,\mathrm{sgn} (t)}=v_j^{\b}$ do not depend on $\sgn(t)$ we say the $w$ is $\delta$-homogeneous.

\begin{definition}[Regular $n$th differential]
\label{REGO} 
Let 
$0\in X$ be  a critical point of corank $\ell\in\{1,\dots, m\}$ of a smooth map
$ F\in C^\infty(X;\R^m)$.
We say that the differential $  \mc D_0^n F:\dom(\mc D_0^n F)\to \coker(d_0F)$, $n\geq 2$, is regular if there exists a separately $\delta$-homogeneous map $w:\R^\ell\to \dom(\mc D_0^n F)$ such that
the function
		$f: \R^\ell\to
		\coker (d_0 F)$
		\begin{equation}\label{faf}
		     f(t) = \mc D_0^n F(w(\varrho(t))   ), \quad t\in \R^\ell ,
		\end{equation}
		is a homeomorphism  with bounded  inverse at zero, i.e.,  there exists  $0<L<\infty$ such that 
		\begin{equation}\label{INVO}
      |f^{-1} (\tau)|\leq L |\tau|,\quad \tau \in\coker(d_0F).
      \end{equation} 
      Above we let $\varrho(t)=\left(\mathrm{sgn}(t_1)|t_1|^{1/n},\dots,\mathrm{sgn}(t_\ell)|t_\ell|^{1/n}\right)$ 
 \end{definition}

When the corank is $\ell=1$ the notion of regular $n$th differential is effective.

\begin{proposition}
\label{REGON}
Let $ F\in C^\infty(X;\R^m)$ be a smooth map such that  
$0\in X$ is a critical point of corank $\ell=1$. Assume that: 
        \begin{itemize}
		\item[i)] $n\geq 2$ is even and there exist $2$ elements $v^{\pm} \in\dom(\mc D_0 ^nF) $ such that $\mc D_0^n F(v^+)>0$ and  $\mc D_0^n F(v^-)<0$; or,
		 \item[ii)] $n\geq 3 $ is odd and there exists $v \in\dom(\mc D_0 ^nF) $ such that $\mc D_0^n F(v)\neq 0$.
		    \end{itemize} 
Then $\mc D_0^nF$ is regular. 
\end{proposition}

\begin{proof} We prove the claim in the case i). Let $v^\pm = (v_1^\pm,\ldots, v_{n-1}^\pm)\in\dom(\mc D_0 ^nF) $ and define the separately $\delta$-homogeneous  function $w:\R\to X^{n-1}$
\[
  w(t) =\left\{
  \begin{array}{l}
   (t v_1^+,t^2 v_2^+, \ldots, t^{n-1} v_{n-1}^+),\quad t>0,
   \\
   (t v_1^-,t^2 v_2^-, \ldots, t^{n-1} v_{n-1}^-),\quad t<0.
  \end{array}
  \right.
\]
By \eqref{DH} we have  $w(t) \in \dom(\mc D_0 ^nF)$ for all $t\in \R$. The function $f(t) =\mc D_0^n F(w(\varrho(t))   )$ is separately linear for $t>0$ and $t<0$, with $f(1)>0$ and $f(-1)<0$. Then it is a homeomorphism from $\R$ to $\R$ with bounded inverse at zero, in the sense \eqref{INVO}.

When $n$ is odd the proof is   analogous.

\end{proof}

 When the corank is larger, $\ell>1$, the existence of (separately) $\delta$-homogeneous maps from $\R^\ell$ into $\dom(\mc D_0^n F)$ is non-trivial.
 Under assumption \eqref{eq:Gohn},
 the following  theorem guaranties that for any linear map into  $\ker(d_0F)$
there exists a $\delta$-homogeneous  extension into $ \dom(\mc D_0^n F)$.

}

\begin{proposition}\label{prop:corrrections}

Let
$F\in C^\infty(X;\R^m)$ be a smooth map such that   $ \mc D_0^h F = 0$ for all $ 2\le h<n$, for some $n\geq 2$.
For any $v_1^1, \ldots,   v_1^\ell \in \ker(d_0F)\subset X$, $\ell\in\N$, there exist vectors $v_j^\b\in X$,  $j=1,\.,n-1$ and $\b\in\I_\ell^0$ with $|\b|=j$,
such that the function $w\in C^\oo(\R^\ell;X^{n-1})$ with coordinates 
\begin{equation} \label{wallo1}
 w_j(t)=\sum_{\b\in \mathscr I_\ell^0,|\b|=j}t^\b v_j^\b,\quad j=1,\ldots, n-1,
\end{equation}
satisfies $w(t)\in \dom(\mc D_0^n F)$ for all $t\in\R^\ell$. In particular, when $\mathrm{e}^i$ is the $i$th vector of the standard basis of $\R^\ell$
we have $v_1^{\mathrm{e} ^i}=v_1^i$.
\end{proposition}

\begin{proof}
        The proof is by induction on $n\geq2$.

        For $n=2$ the statement  follows from the fact that $\dom(\mc D_0^2F)=\ker(d_0F)$ is a vector space.
        In this case, we   have $j=1$ and $\beta=\eee^ i$ for some $i$. Fixing  $v_1^\b = v_1^i$ with  $v_1=(v_1^1,\ldots, v_1^\ell)$,
          formula \eqref{wallo1}        gives a function $w_1$ with values in $\dom(\mc D_0^2F)$.

        We assume the claim  for $n-1$ and we   prove it for $n$.  In particular, for $j\leq n-2$,
        the vectors $v_j^\b\in X$ are already fixed so that the functions defined in
        \eqref{wallo1} with $j\leq n-2$ satisfy
        $(w_1(t),\.,w_{n-2}(t))\in\dom(\mc D_0^{n-1}F)$ for all $t\in\R^\ell$.
        Our goal is to   find $v_{n-1}^\b$, for $\b\in \mathscr I_\ell^0$ with $|\b|=n-1$, such that  $w(t)=(w_1(t),\.,w_{n-1}(t))\in\dom(\mc D_0^{n}F)$.

        The condition $w(t)\in\dom(\mc D_0^{n}F)$ is equivalent to
        \begin{itemize}
         \item[1)] $(w_1(t),\.,w_{n-2}(t))\in\dom(\mc D_0^{n-1}F)$;
         \item[2)] $D_0^{n-1}F(w(t))=0$.
        \end{itemize}
        The first condition is true by induction. By formula \eqref{FUF}, the latter
        is equivalent to
        \be\label{eq:pippo}
                d_0F(w_{n-1}(t))+\sum_{h=2}^{n-1}\sum_{\alpha\in \mathscr I_h,|\alpha|=n-1}\frac{(n-1)!}{h!\alpha!}d_0^hF(w_\alpha(t))=0.
        \ee
        We  solve this equation to determine the vectors $v_{n-1}^\b\in X $. By linearity, we have
        \[
         d_0F(w_{n-1}(t))=\sum_{\b\in \mathscr I_\ell^0,|\b|=n-1}t^\b d_0F(v_{n-1}^\b),
        \]
        and
        \begin{align*}
         d_0^hF(w_\a(t))&=d_0^hF(w_{\a_1}(t),\.,w_{\a_1}(t))\\
         &=\sum_{\b^1\in \mathscr I_\ell^0,|\b^1|=\a_1}\.\sum_{\b^h\in \mathscr I_\ell^0,|\b^h|=\a_h} t^{\b^1+\.+\b^h}d_0^hF(v_{\a_1}^{\b^1},\.,v_{\a_h}^{\b^h}).
        \end{align*}
        By  the identity principle of  polynomials, solving equation \eqref{eq:pippo} is equivalent to  solving the set of equations
        \begin{equation}\label{ECCO}
         d_0F(v_{n-1}^\b)+\sum_{h=2}^{n-1}\sum_{\substack{\alpha\in \mathscr I_h,|\alpha|=n-1\\ \b^1+\.+\b^h=\b}}\frac{(n-1)!}{h!\alpha!}d_0^hF(v_{\a_1}^{\b^1},\.,v_{\a_h}^{\b^h})=0,
        \end{equation}
        for $\b\in\mc I_\ell^0$ with $|\b|=n-1$. This is possible since, by assumption, we have  $\mc D_0^h F = 0$ for $1\le h\le n-1$. This implies that $\proj(d_0^hF(v_{\a_1}^{\b^1},\.,v_{\a_h}^{\b^h}))=0$, i.e.,  $d_0^hF(v_{\a_1}^{\b^1},\.,v_{\a_h}^{\b^h})\in \mathrm{im}(d_0F)$ and thus we can find
        $v_{n-1}^\b\in X$ solving   equation \eqref{ECCO}.
\end{proof}

{\color{black} Proposition \ref{prop:corrrections} can be improved making the construction   separately $\delta$-homogeneous. We omit the details.

}

\section{Open mapping theorem of order $n$}\label{sectwo}

In this section, we prove our main  open mapping theorem.

 {\color{black}

\begin{theorem}\label{thm:Gopen}
Let $ F\in C^\infty(X;\R^m)$ be a smooth map  with regular differential $ \mc D_0^n F$, $n\geq 2$,  at the critical point $0\in X$.
 Then $F$ is open at $0$.
\end{theorem}

}

\begin{proof}  Let $0\in X$ be a critical point for $F$  of corank $\ell\in\{1,\dots,m\}$.
   We identify $\coker(d_0 F)=\R^\ell$ and we split $X=\R^{\ell-m} \oplus \ker(d_0F)$.

   {\color{black}
   The regularity of    $ \mc D_0^n F$  means that there exists  a separately $\delta$-homogeneous map  $w:\R^\ell \to \dom(\mc D_0^n F)$  
      such that  the function $f: \R^\ell\to\R^\ell $ in \eqref{faf} is a homeomorphsim and satisfies \eqref{INVO}.
   By formula \eqref{wallo}, the map $w =(w_1,\dots,w_{n-1})$ is of the form 
	\[
	w_j(t)= \sum_{\b\in \mathscr I_\ell^0,|\b|=j} t^\b v_j^{\b,\sgn(t)} ,\quad j=1,\dots,n-1,
	\]
	for some $v_j^{\b,\sgn(t)}\in X$.
	}

   	We define the map $\Phi:\R^{m-\ell}\times \R^\ell\to X$ letting  
    \[
		\Phi(r,t)= r +  \sum_{j=1}^{n-1} \frac{w_j(t)}{j!} ,
		\quad 
		(r,t)\in \R^{m-\ell}\times \R^\ell.
	\]
Above and hereafter, we identify $r\in \R^{\ell-m}$ with $(r,0)\in  X$, so that the sum $r+v$ with $v\in X$ is well defined.
We   claim that we have the  expansion
	\begin{equation}\label{expa}
			F\left(\Phi(r,t)\right)=
			d_0F(r)+D_0^n F( w (t),0) + R(r, t),
	\end{equation}
where the remainder  satisfies 
\be\label{eq:remlow2}
	\lim_{(r,t)\to 0} \frac{R(r,t)}{|r|+|t|^n} = 0.
\ee

For any positive scalar $s\geq0$, using the homogeneity $w_j(st)= s^j w(t)$ we obtain the formula 
  \[
  \Phi(0,s t)= 
  \sum_{j=1}^{n-1}\frac {s^j}{j!} w_j(t),
  \]
and, for fixed $t$, the function $\phi(s) =F(\Phi(0,st))$ has the Taylor expansion
  \begin{equation} \label{sis}
       \phi(s) =\sum_{j=1}^n \frac{\phi^{(j)} (0) }{j!} s^j + \frac{\phi^{(n+1)} (\bar s) }{(n+1) !}  s^{n+1} ,\quad s\in[0,1] ,
  \end{equation}
  for some $\bar s \in [0,s]$.
  
  By   definition \eqref{DOF}, we have  $\phi^{(j)} (0)  = D^j_0 F(w_1(t),\dots,w_{j-1}(t))$ and since   $w(t) \in \dom(\mc D_0^n F)$, we deduce that 
$
\phi^{(j)} (0) =0$ for $j=1,\dots, n-1$, while for $j=n$ we have
\[
\phi^{(n)} (0) = D^n_0 F(w_1(t),\dots,w_{n-1}(t),0)=D^n_0 F(w(t),0).
\]
From the Taylor expansion \eqref{sis} with  $s=1$, we obtain
\begin{equation}
\label{ERRO}
F(\Phi(0,t) )= D^n_0 F(w(t),0 ) +   E(t),\quad t\in\R^\ell,
\end{equation}
where the error satisfies the estimate
\begin{equation}
\label{ERROS}
  |E(t)| =\Big|  \frac{\phi^{(n+1)} (\bar s) }{(n+1) !} \Big| \leq C|t|^{n+1},
\end{equation}
for some constant $C>0$.

Now, we obtain the expansion \eqref{expa} adding  a development in $r$ of the first order.
We are left with the proof of  \eqref{eq:remlow2}. Also by \eqref{ERROS}, the error $R(r,t)$ is estimated by a sum of the form
\[
| R(r,t)|\leq  \sum_{0\leq i \leq 2, 0\leq j \leq n+1  } c_{ij} |r|^i |t|^j,
\]
with constants $c_{ij}$ such that  $c_{0j}=0$ if $j\leq n$ and $c_{10}=0$. So we have $| R(r,t)|\leq  C(|r|^2 +|r| |t| +|t|^{n+1})$ and   the mixed term is estimated by Young inequality:
\[
 |r| |t|\leq \frac{n}{n+1} |r|^{(n+1)/n} + \frac{1}{n+1}   |t|^{n+1}.
\]
This finishes the proof of  \eqref{eq:remlow2}.

The map $F$ is open at $0$ if the map   $F\circ\Phi$ is open at $0$. And $F\circ\Phi$ is open at $0$
 if and only if the map
\[
(r,t)\mapsto 
\Psi(r,t)= F( \Phi (r,\varrho(t)))= d_0 F(r) + D_0^n F (w(\varrho(t)),0)  +R(r,\varrho(t))
\]
 is open at $(r,t)=0$. We will show that $\Psi$ is open at zero by a fixed point argument.

 With respect to the factorization   $(r,t) \in \R^{m-\ell}\times \R^\ell$, we introduce the norm $\| (r,t)\| = \max\{ |r|, \lambda_0 |t|\}$ and the balls \[
 B_\delta =\{
 (r,t) \in \R^{m-\ell}\times \R^\ell: \| (r,t)\|\leq \delta\}
 \]
  for positive $\delta>0$. The balls $B_\delta$ are compact and convex.
 The parameter $\lambda _0>0$ will be fixed later.
 
 The map $\Psi$ is open at $0$ if for any (small) $\varepsilon>0$
 there exists $\delta>0$ such that $B_\delta \subset \Psi(B_\varepsilon)$. 
 We pick $(\xi,\tau) \in B_\delta$ and we look for $(r,t)\in B_\varepsilon$ such that $\Psi(r,t) =(\xi,\tau)$.
 We factorize
 \[
 D_0^n F (w(\varrho(t)),0)  = (\mc D_0^n F (w(\varrho(t))), g(t)) = (f(t), g(t)),
 \]
 and $R(r,\varrho(t))= (R_1 (r,t),R_2(r,t) )\in \R^{m-\ell}\times \R^\ell$. Here, with a slight abuse of notation, we are incorporating $\varrho$ into $R_1$ and $R_2$.

 Since $g$ is continuous and homogeneous of degree $1$, there exists a constant $C_1>0$ such that
 \begin{equation}
 \label{piffero}
   |g(t)| \leq C_1 |t|.
 \end{equation}
 By  \eqref{eq:remlow2}, for any $0<\sigma\leq 1$ there exists an $\varepsilon>0$ such that for $|r|+|t|\leq \varepsilon$ (in particular for $(r,t) \in B_\varepsilon$)
 we have
 \begin{equation}
 \label{vix}
 |R_1 (r,t)|+| R_2(r,t)|\leq \sigma( |r|+|t|).  
 \end{equation}
 We will fix $\sigma$ in a while.

 Equation 
 $\Psi(r,t) =(\xi,\tau)$ is then equivalent to the system
 \begin{equation} \label{SOS}
      \left\{
      \begin{array}{ll}
          d_0 F(r) + g(t) +R_1(r,t) = \xi
          \\
          f(t) +  R_2(r,t) =\tau,
      \end{array}
      \right.
 \end{equation}
that reads as the following fixed-point system
 \[
      \left\{
      \begin{array}{ll}
          r = d_0 F^{-1}\big (\xi - g(t)-R_1(r,t) \big)=h_1(r,t)
          \\
          t = f^{-1}  \big( \tau- R_2(r,t)\big)= h_2(r,t).
      \end{array}
      \right.
 \]
 We claim that the map $h = (h_1,h_2)$ maps $B_\varepsilon$ into 
 itself, provided that $\delta>0$, $\sigma>0$,  and $\lambda_0>0$ are small enough.
 Indeed, by \eqref{piffero} and \eqref{vix} we have
 \[
 \begin{split}
 |h_1(r,t)| & \leq \| d_0 F^{-1}\| (|\xi| + |g(t))|+|R_1(r,t)|)
 \\
 &
  \leq  C_2 ( |\xi| +|t| + \sigma   (|r|+|t|))
  \\
  & \leq C_2  (\delta + 2\lambda_0 \varepsilon +\sigma \varepsilon ). 
  \end{split}
 \]
 Choosing $\delta \leq \varepsilon/3 C_2$, $ \lambda_0=1/6 C_2$, and $ \sigma\leq 1/3C_2$ we obtain  $ |h_1(r,t)| \leq \varepsilon$.

 On the other hand, by \eqref{INVO} and \eqref{vix}
 \[
 \begin{split}
   |h_2(r,t)|& \leq L (|\tau|+|R_2(r,t)|)
   \\
   &
    \leq L (|\tau|+ \sigma (|r|+|t|)) 
    \\
    &\leq L ( \lambda_0 \delta + 2\sigma  \varepsilon ).
     \end{split}
 \]
 Choosing $\delta \leq \varepsilon/2L$ and $ \sigma \leq \lambda_0 /4L$ we obtain $   |h_2(r,t)|\leq \lambda_0 \varepsilon$.
 This finishes the proof that for each $(\xi,\tau) \in B_\delta$ there exists $(r,t)\in B_\varepsilon$ solving  the system \eqref{SOS}.

\end{proof}

When the corank is $\ell=1$, by Proposition~\ref{REGO} we have the following:

\begin{corollary}
\label{REGONO}{\color{black}
Let $ F\in C^\infty(X;\R^m)$ be a smooth map such that  
$0\in X$ is a critical point of corank $\ell=1$. }Assume that: 
        \begin{itemize}
		\item[i)] $n\geq 2$ is even and there exist $2$ elements $v^{\pm} \in \dom(\mc D_0^n F)$ such that 
		$\mc D_0^n F(v^-) $ and $\mc D_0^nF(v^+)$
		have opposite sign; or,
		 \item[ii)] $n\geq 3$ is odd and there exists $v \in \dom(\mc D_0^n F)$ such that $\mc D_0^nF(v)\neq 0$ .
		    \end{itemize} 
Then $F$ is open at $0$.
\end{corollary}

\section{Integrals on simplexes}\label{secthree}

In this section, we prove some elementary properties of integrals on simplexes that will be used in the analysis of the end-point map.
Here and hereafter, $I=[0,1]$ denotes the unit interval. We fix $d\in\N$ (it will be the rank of the distribution of vector fields on the manifold) and in the rest of the paper we let 
\[
    X = L^2(I;\R^d).
\]

The tensor product  $\otimes: \R^\ell \times \R^m\to \R^{\ell m}$ is defined by 
	\[
		(v\otimes w)^k = v^i w^j, \quad   k =m(i-1)+j  ,
	\]
where $ 1\le i\le \ell $ and $1\le j\le m$.  Above,  we are using the notation $v= (v^1,\dots, v^\ell)\in\R^\ell $, etc.
The map $\otimes$ is associative but not commutative.

\begin{definition}\label{defi:simplexts}
	For  $n\in \N$ and $t,s \in I$  such that $t+s\le 1$, we  define the $n$-dimensional simplex
\be\label{eq:simplextsf}
	\Sigma_n(t,s)=\big\{ (t_1,\dots,t_n)\in I^n\mid t<t_n<\dots<t_1<t+s \big\}.
\ee
When $t=0$ and $s=1$ we use the short notation $\Sigma_n=\Sigma_n(0,1)$.
We also let  
\be\label{eq:simplexts}
	\Sigma_n^\flat(t,s)=\big\{ (t_1,\dots,t_n)\in I^n\mid t<t_1<\dots<t_n<t+s \big\},
\ee
and   $\Sigma_n^\flat=\Sigma_n^\flat(0,1).$

\end{definition}

For   $n\in \N$ we define the   subset of $X $
\be\label{eq:workingspacen}
	U_n =\left\{ v\in X \mid \int_{\Sigma_n}v(t_n)\otimes \dots\otimes v(t_1) d\L^n =0 \right\},
\ee
Here and in the following, we denote by $\L^n$ the Lebesgue measure on $\R^n$. We also set  
\begin{equation} \label{bip}
	V_n= \bigcap_{i=1}^n U_i.
\end{equation}
 For any multi-index $\a\in \I_{n,d} =\big \{\a\in\I_n\mid \a_i \in\{1,\dots,d\}\big\}$ and $v\in X$, we define the integral
 \[
    I_n^\a(v) =\int_{\Sigma_n} v^{\a_n}(t_n) \dots v^{\a_1}(t_1) d\L^n.
 \]
 Then $v\in U_n$ if and only if $I_n^\a(v) =0$ for all $\a\in\I_{n,d}$.

For $v\in X $, $n\in\N$ and $t,s \in  I $ such that $t+s\leq 1$, we  let
\[
\begin{split}
 I_n(t,s;v)& = 
 \int_{\Sigma_n(t,s) }v(t_n)\otimes \dots\otimes v(t_1) d\L^n,
 \\
 I_n^\flat (t,s;v)& = 
 \int_{\Sigma_n^\flat (t,s) }v(t_n)\otimes \dots\otimes v(t_1) d\L^n.
\end{split}
\]

\begin{lemma}\label{bemolle} For any $v\in  V_n$ and $ t\in  I $ we have 
\begin{equation}\label{folx}
 I_n^\flat (0,t;v) = (-1)^n  
 I_n (t,1-t;v).
\end{equation}

\end{lemma}

\begin{proof}
The proof is by induction on $n\in\N$. When $n=1$ the claim reads
\[
  \int_0^t v(s) ds = - \int_t^1 v(s)d s, \quad t\in  I ,
\]
that holds true because $v\in V_1$ means $\displaystyle \int_0^1 v(s) ds=0$.

We assume that  formula \eqref{folx} holds for $n-1$ and we prove it for $n$. Indeed, using first  $v\in U_n$   and then $v \in V_{n-1}$ we get
\[
\begin{split}
 I_n^\flat (0,t;v) &   = \int_{0}^ t v(t_n) \otimes  I_{n-1} ^\flat (0,t_n;v) 
 d t_n 
 \\
 &   =-  \int_{t}^ 1 v(t_n) \otimes  I_{n-1}^\flat  (0,t_n;v) 
 d t_n 
 \\
 &=(-1)^{n}  \int_{t}^ 1 v(t_n) \otimes  I _{n-1} (t_n,1 -t_n;v) 
 d t_n
 \\
 &=(-1)^{n}  I_n (t,1-t;v).
  \end{split}
\]
\end{proof}

The reverse parametrization of a function $v \in X $ 
is the function $v^\flat \in X $ defined by the formula
\[
v^\flat(t)= v(1-t),\quad     t\in I.
\]

\begin{corollary}\label{lemma:finitevs}
	Let 
	$v \in X $. Then  $v\in V_n$ if and only if  $v^\flat \in V_n$.
\end{corollary}

\begin{proof}
If $v\in V_n$, by Lemma~\ref{bemolle} with $t=0$ it follows that $v^\flat \in V_n$.
The opposite implication follows from $v^{\flat\flat}=v$.
\end{proof}

The set $V_n$ is  stable   with respect to localization.
Given $v\in X $, $s>0$ and $t_0\in I$ such that $t_0+s\leq 1$, we define 
 \begin{equation}\label{vs}
  v_{t_0,s}(t) =v\left(\frac{t-t_0}{s}\right) \chi_{[t_0,t_0+s]}(t) ,\quad   t\in I.
 \end{equation}

\begin{lemma}\label{lemma:decomposeN}
	If $v\in V_n$ then  $ v_{t_0,s}\in V_n$ for all $s>0$ and $t\in I$ such that   $ t_0+s\leq 1 $. 
\end{lemma}

\begin{proof}
The claim $ v_{t_0,s} \in X$ is clear. We prove that, for every $1\le i\le n$,
	\be\label{eq:rag}
	I _i ( t_0,s ;v_{t_0,s}) = 	\int_{\Sigma_i( t_0,s) } v_{t_0,s}(t_i) \otimes\dots\otimes v_{t_0,s}(t_1)d\L^n=0.
	\ee
	Indeed, by the change of variable  $
		(t_1,\dots, t_i)= (s \tau _1+t_0  ,\dots, s \tau _i+t_0)$, we get 
			\[
		I _i ( t_0,s ;v_{t_0,s})  = s^i  I_i(0,1;v)=0 .
	\] 
\end{proof} 
The set $V_{n-1} \subset X$ is not a linear space and the map $v\mapsto I_n(v) = I_n(0,1;v)$ is not additive. 
However, we can construct linear subsets of $V_{n-1}$  of any finite dimension starting from one function.
Given  $v\in X$, we define $v_1,v_2\in X$ letting
\[
  v_1 = v_{0,1/2}\quad \textrm{and}\quad
  v_2 = v_{1/2,1/2}.
\]
These are the localization of $v$ with $t_0=0,1/2$ and $s=1/2$.

\begin{proposition} \label{ITER} If $v\in V_{n-1}$ then $v_1,v_2,v_1+v_2\in V_{n-1}$ and
\[ 
   I_n(v_1+v_2)= I_n(v_1)+  I_n(v_2).
\]
Moreover, we have  $ I_n(v_1)=  I_n(v_2) = \frac{1}{2^n} I_n(v)$.
\end{proposition}

\begin{proof} The fact that $v_1,v_2\in V_{n-1}$ is proved in Lemma~\ref{lemma:decomposeN}.
We show the remaining claims. For any multi-index $\a\in \I_{h,2}$, $1\leq h\leq n-1$, consider the integral 
\begin{equation}\label{tk}
  I^\a(v_1,v_2) = \int _{\Sigma _h} v^1_ {\a_1}(t_1)    \dots v^h_ {\a_h}(t_h)d\L^h.
\end{equation}
Letting $I_1=[0,1/2]$ and $I_2 =[1/2,1]$, the support of the function $v^1_ {\a_1}(t_1)\dots  v^h_ {\a_h}(t_h) $ is contained in the product 
$I_{\a_1}\times \dots\times  I_{\a_h}$. If there exist $i<j$ such that $\a_i= 1<  \a_j =2$, then 
$
\Sigma_h \cap I_{\a_1}\times \dots\times  I_{\a_h}= \emptyset$, and then $I^{\a}(v_1+v_2)=0$.

The complementary case is when $\a_1=\dots=\a_{k} = 2$ and $\a_{k+1}=\dots=\a_h=1$ for some $k=0,1,\dots, h$.
In this case, 
the integral in \eqref{tk}  splits
into the product of two integrals: 
\[
  I^\a(v_1,v_2) = \Big( \int_{\Sigma_ k }  v^1_ {2}(t_1)    \dots v^k_ {2}(t_k)d\L^k\Big) 
  \Big( \int _{\Sigma_{h-k}} v^{k+1} _ {1}(t_{k+1})    \dots v^h_ {1}  (t_h)d\L^{h-k}\Big) .
\]
If $v_1,v_2\in V_{n-1}$ this shows that $ I^\a(v_1,v_2) =0$ for all $\a \in \I_{h,2}$ and for all $h\leq n-1$.
This proves that $v_1+v_2\in V_{n-1}$.

When $h = n$ the argument above shows that for all $\a\in \I_{n,2}$ such that $\a\neq (1,\dots,1)$ and $\a\neq (2,\dots,2)$ we have
$ I^\a(v_1,v_2)=0$. We conclude that
\[
I_n (v_1+v_2) = \sum_{\a \in \I_{n,2} } I^\a(v_1,v_2) = I_n(v_1) + I_n(v_2).
\]
 \end{proof}

\section{Expansion of the end-point map}

\label{EXPA}

\renewcommand{\odot}{\circ}

In this section we expand the end-point map and we compute  its $n$th order term. The computations use the language of chronological calculus for non-autonomous vector fields. A detailed introduction to this formalism  can be found in \cite[{\color{black} Chapter 2}]{AgrSac}. 
 {\color{black} A different approach to the Taylor expansion of the end-point  map in Lie groups   using adapted coordinates is given in
\cite{JozB21}.}

Let $M$ be a manifold with dimension $m=\mathrm{dim}(M)$. Since   our analysis is local, we shall without loss of generality identify $M$ with $\R^m$.
So $M$-valued maps will be in fact $\R^m$-valued, fitting the framework of Section~\ref{sec:one}.

For vector fields $f_1,\.,f_d\in$ Vec$(M)$ and $u\in X $, we define the time-dependent vector field $f_{u(t)}=\sum_{i=1}^d u_i(t)f_i$. For a fixed initial point $q\in M$, the end-point map $F_{q}:X \to M$ is defined as  
\[
 F_{q}(u) =q\odot \eexp\int_0^1f_{u(t)}dt,\quad u\in X .
\]
Here,   $\eexp\int_0^1f_{u(t)}dt$ denotes the right exponential of a time-dependent vector field. As explained in  \cite[Chapter 20]{AgrSac}, 
points of $M$, vector fields, and diffeomorphisms of $M$ are identified with operators on $C^\infty(M)$. In this formalism, 
$\odot$ stands for a composition of operators.

We denote by $\q =F_{q}(u)$  the end-point and we define the map
 $G^u_{\q }:X  \to M$ letting 
\[
    G^u_{\q }(v) = \q \odot \eexp\int_0^1 g_{v(t)}^{u,t}dt,\quad v\in X ,
\]
where $ g_{v(t)}^{u,t}$ is the time-dependent vector field 
\begin{equation}
\label{gufo}
 g_{v(t)}^{u,t}=\Ad\left(\eexp\int_1^t f_{u(\tau)}d\tau\right)f_{v(t)}.
\end{equation}
The maps $F_q$ and $G^u_{\q }$ are related by the variation formula, see \cite[Formula $(2.28)$]{AgrSac}.
For 
$u,v \in X $ we have 
\begin{align*}
 F_{q}(u+v)&=q \odot \eexp\int_0^1f_{u(t)+v(t)}dt\\
 &=q\odot \eexp\int_0^1f_{u(t)}dt \odot \eexp\int_0^1\Ad\left(\eexp\int_1^t f_{u(\tau)}d\tau\right)f_{v(t)}dt
 \\
 & =  G^u_{\q }(v).
\end{align*}
For the definition of the $\Ad$ operator in chronological calculus see  \cite[Chapter 2]{AgrSac}.

The control  $u$ is a critical point of  corank $\ell$
for $F$ if and only if   0 is a critical point of $ G^u_{\q }$ of corank
$\ell$. We shall omit the subscript $\q $ and the superscript $u$ and write $G=G^u_{\q }$. We call $G$  \emph{variation map}.

Our next goal is to compute the Taylor's expansion of the variation map. For $k\in\N$ and $v\in X $, we define the vector field $W_k(v)$   as 
\be
 \label{ciao}
 \begin{split}
 W_k(v) & =\int_{\Sigma_k} \ad g_{v(\tau_k)}^{u,\tau_k} \circ\.\circ \ad g_{v(\tau_2)}^{u,\tau_2}\big(g_{v(\tau_1)}^{u,\tau_1}\big)d\L^k 
 \\
 & =\int_{\Sigma_k} [ g_{v(\tau_k)}^{u,\tau_k},\dots,   g_{v(\tau_1)}^{u,\tau_1}] d\L^k .
\end{split}
\ee
Here and hereafter, we use  the following notation for the iterated bracket of vector fields $g_k,\dots,g_1$:
\[
[g_k,\dots, g_1] = [ g_k , [\cdots,[g_2,g_1]\cdots]]=\ad g_ k \circ \cdots \circ\ad g_2 (g_1).
\]

For a multi-index   $ \beta \in\I_k$ we define the operator (composition of vector fields)
\begin{equation}\label{BOX}
 W_\beta (v)=W_{\beta_1}(v)\odot\.\odot W_{\beta_k}(v).
\end{equation}
The  operator-valued map  $v\mapsto W_\b (v)$ introduced in \eqref{BOX}
is polynomial in $v$ with homogeneous degree $p=|\b|$. Its $p$-polarization is defined via the formula
\begin{equation}\label{WIWA}
W_\b (v_1,\dots,v_p) =\frac{1}{p!} \frac{\partial ^p}{\partial t_1 \dots \partial t_p} 
 W_\b \Big(\sum_{i=1}^p t_i v_i\Big)\bigg| _{t_1=\dots=t_p=0}, 
\end{equation}
where $ v_1,\dots,v_p \in X $.
This definition is consistent with  \eqref{buio}.

By  the   argument of Lemma 3.3 and Remark 3.4 in \cite{BMP20},   
for any $p\in \N$ and for any $v\in X$ the $p$-differential of $G$ has the representation
\begin{equation}
\label{luce}
 d_0^pG(v)=
 \sum_{k=1}^p\sum_{\b\in\I_k, |\b|=p}  c_\b   W_{	\b }(v),
\end{equation}
where, for any $\b\in\I_k$, we set  
\[
 c_\b = |\b| !  \prod_{s=1}^k(\b_1+\.+\b_s)^{-1}\in \R.
 \]
\noindent Using these formulas we obtain a representation for the differentials $D^h_0G$.

%
%
%
%
%

\begin{lemma} For any $h\in\N$ and for all 
 $ v=(v_1,\dots,v_h) \in X ^h$ we have
\begin{equation} \label{pluto}
 D_0^hG(v) =\sum_{p=1}^h\sum_{\a\in\I_p, |\a|=h}\frac{h!}{\a!p!}\sum_{k=1}^p\sum_{\b\in\I_k, |\b|=p}
  c_\b  W_\b( v_\a),
\end{equation}
where $v_\a=(v_{\a_1},\dots,v_{\a_p})$ for $\a\in \I_p$.

\end{lemma}

\begin{proof}
Formula \eqref{FdB} reads
\begin{equation}\label{pluto3}
 D_0^hG(v )=\sum_{p=1}^h\sum_{\a\in\I_p, |\a|=h}\frac{h !}{\a!p!}d_0^pG(v_\a),
\end{equation}
and by \eqref{buio}, \eqref{luce}, and \eqref{WIWA} we deduce that,   for $w=(w_1,\dots,w_p)\in X ^p$, 
\begin{equation}\label{PLAT}
\begin{split} 
d_0^p G(w) & =\frac{1}{p!} \frac{\partial ^p}{\partial t_1 \dots \partial t_p} 
 d_0^pG\Big(\sum_{i=1}^p t_i w_i\Big)\bigg| _{t_1=\dots=t_p=0}
 \\
 & =\frac{1}{p!}
  \sum_{k=1}^p\sum_{\b \in\I_k, |\b|=p}  c_\b \frac{\partial ^p}{\partial t_1 \dots \partial t_p} 
 W_{\b}\Big( \sum_{h=1}^p t_h w_h\Big)\bigg| _{t_1=\dots=t_p=0}
 \\
 &
 =
  \sum_{k=1}^p\sum_{\b\in\I_k, |\b|=p}  c_\b   W_{\b}(w) .
 \end{split}
\end{equation}

\end{proof}
%

For  a given $v\in X $ let us consider the localization $v_{t_0,s}$ for some $t_0\in[0,1)$ and 
small $s>0$, as in \eqref{vs}.

\renewcommand{\0}{\vartheta}

\begin{proposition} \label{PIX} 
Let $h\in\N$,  $v\in X $, and  $ t_0\in (0,1)$. For any $s \in (0, 1-t_0)$ we have   
  \be
  \label{Wn}
  W_h( v_{t_0,s})=
  s^h\int_{\Sigma_h} [ g_{v(t_h)}^{t_0},\dots,  g_{v(t_1)}^{t_0}] d\L^h + O(s^{h+1}).
 \ee
 Moreover, there exists a constant $C>0$ such that $ |O(s^{h+1})|\leq C s^{h+1}$ for all $v\in X$ with $\| v\|_X\leq 1$.
\end{proposition}

\begin{proof} With the notation introduced in \eqref{gufo} and omitting the superscript $u$, we have
 \[
 \begin{split} 
  g_{v_{t_0,s}(t)}^\tau& = \Ad\left(\eexp\int_1^\tau f_{u(\s)}d\s\right)f_{v_{t_0,s}(t)}
  \\
  &
  = \sum_{i=1}^d v_{t_0,s}^i(t)\Ad\left(\eexp\int_1^\tau f_{u(\s)}d\s\right)f_i
  \\
  &
  = \sum_{i=1}^d v_{t_0,s}^i(t)g^\tau_i,
 \end{split}
 \]
 where $g^\tau_i$ is defined via the last identity.  Letting, for  $\a\in  \I_{h,d} $,
 \begin{equation} \label{TOX}
 J_{t_0,s}^{\a  }= \int_{\Sigma_h  (t_0,s)} v_{t_0,s}
 ^{\a_h}(\tau_h)\.v_{t_0,s}^{\a_1}(\tau_1)[g^{\tau_h}_{\a_h },\dots ,g^{\tau_1}_{\a_1}] d\L^h,
 \end{equation}
 formula \eqref{ciao} reads 
 \[
  W_h(v_{t_0,s}) 
  = \sum_{\a\in \I_{h,d} } J_{t_0,s}^{\a}.
\]

\noindent
With  the change of variable
 $\0_i=\frac{\tau_i-t_0}{s}$, for $i=1,\.,n$, the  integral in \eqref{TOX} becomes
  \begin{equation} \label{TOX1}
 J_{t_0,s}^{\a  }= s^ h \int_{\Sigma_h } v
 ^{\a_h}(\0_h)\.v^{\a_1}(\0_1)[g^{s\0_h+t_0}_{\a_h },\dots ,g^{s\0 _1+t_0 }_{\a_1}] d\L^h.
 \end{equation}
 Since the maps
 \[
  t\mapsto g^t_i=\Ad\left(\eexp\int_1^t f_{u(\s)}d\s\right)f_i, \quad i=1,\.,d,
 \]
 are Lipschitz continuous,   for every $i=1,\.,d$  and $j=1,\.,h$ we have the expansion
 \[
  g_i^{s\0_j+t_0}=g_i^{t_0}+O(s),
 \]
with a uniform error   $O(s)$ for $\0_j\in I $. So we conclude that
 \begin{align*}
  J_{t_0,s}^{\a }   =   s^h[g_{\a_h }^{t_0}, \.,g_{\a_1}^{t_0}]\int_{\Sigma_h }v^{\a_h}(\0_h)\.v^{\a_1}(\0_1)d\L^h + O(s^{h+1}).
 \end{align*}
 The claim \eqref{Wn}  follows by summing over     $\a\in \I_{h,d}$.
 \end{proof}

\begin{corollary} \label{CORO}

Let $v\in V_h$ for some $h\in \N$  and  $ t_0\in (0,1)$. For any $s \in (0, 1-t_0)$ we have  
 $d_0^h G (v_{t_0,s},\dots,v_{t_0,s}) = O(s^{h+1})$.
\end{corollary}

\begin{proof}  
By formula \eqref{PLAT}, the $h$-differential of $G$ has the representation
\begin{equation}\label{PLAT3}
d_0^h G(w) 
 =
  \sum_{k=1}^h\sum_{\b\in\I_k, |\b|=h}  c_\b   W_{\b}(v_{t_0,s},\dots,v_{t_0,s}) .
\end{equation}
Let $\b\in\I_k$ with $ |\b|=h$.
We claim that the coefficient of $s^h$ in the expansion of $s\mapsto W_{\b}(v_{t_0,s},\dots,v_{t_0,s})$ vanishes.
Indeed,  consider the coordinate $j=\b_i $. By 
Proposition~\ref{PIX} we have
\[
\begin{split} 
   W_{j}(v_{t_0,s})& = s^{j}\sum_{\a\in \I_{j,d} }
   [g_{\a_j }^{t_0}, \.,g_{\a_1}^{t_0}]\int_{\Sigma_j}v^{\a_j}(\0_j)\.v^{\a_1}(\0_1)d\L^j + O(s^{j+1})
   \\
   &
 = O(s^{j+1}),
 \end{split}
\]
because for $j\leq h$ we have 
\[
\int_{\Sigma_j }v^{\a_j}(\0_h)\.v^{\a_1}(\0_1)d\L^j=0, 
\] 
by our assumption $v\in V_h$ and by Lemma~\ref{bemolle}.
The claim follows. 

 \end{proof}

{\color{black}

Assume that for small $s>0$ we have $w^s=(w_1^s,\ldots, w_{n-1}^s) \in \dom(\mc D_0^n G)$ where $w_1^s =  v _{t_0,s}$ for some   $v\in X$ and $t_0\in (0,1)$.

}

 \begin{lemma}\label{LEOX}
 If $v\in V_{n-1}$ then $\| w_{j}^s \|_{X }=O (s^{j+1})$, $s\to0^+$, for all $j=2,\dots,n-1$.  
 \end{lemma}

\begin{proof}
The proof is by induction on $j=2,\dots,n-1$. We start with $j=2$.
Since $(w_1^s, w_2^s)\in\dom(\mc D_0^3G)$ we have $D_0^2G(w_1^s, w_2^s)=0$, and by \eqref{FdB} this equation 
reads
\[
  d_0 G( w_2^s)= - d_0 ^2G (v_{t_0,s}, v_{t_0,s}) = O (s^3),
\]
by Corollary~\ref{CORO}. The claim follows composing with the inverse of   $d_0G$.

Now we assume that the claim holds for $j\leq n-2$ and we prove it for $j=n-1$. 
Since $w^s  \in\dom(\mc D_0^nG)$ we have $D_0^{n-1} G(w^s )=0$ and, by \eqref{FdB}, this equation 
reads
\[
  d_0 G( w_{n-1} ^s)= - d_0 ^{n-1}G (v_{t_0,s}, \dots, v_{t_0,s})
  -\sum_{h=2}^{ n-2}  \sum_{\a\in\I_h,|\a|=n-1} \frac{(n-1)!}{\a! h!} d_0^h G \big(w^s_\a\big). 
\]
We   have $d_0 ^{n-1}G (v_{t_0,s}, \dots, v_{t_0,s})=O(s^n)$, by Corollary~\ref{CORO}.

We estimate the terms in the sum.
When $2\leq h\leq n-2$ and $\a\in \I_h $ with $|\a|= n-1$,  the multi-index $\a$ contains at least one coordinate different from $1$ and does not contain the coordinate $n-1$, and so  
\[
  \mathrm{Card} \{ j \mid \a_j=1\}+\sum_{i=2}^ {n-2}  (i+1) \,  \mathrm{Card} \{ j \mid \a_j=i\}>|\a| = n-1.
\]
Then,  from our inductive assumption it follows that $d_0^h G \big( w^s_\a\big)= O(s^n)$.  

\end{proof}
 
\begin{lemma} Assume that for small $s>0$ we have $w^s=
(w_1^s,\ldots, w_{n-1}^s) \in \dom(\mc D_0^n G)$, with $w_1^s =  v _{t_0,s}$ for some $v\in X$ and $t_0\in (0,1)$.
 If $v\in V_{n-1}$ then we have  for $s\to0^+$
\begin{equation} \label{pluto5}
 D_0^nG(w^s) =  c_n  s^n  \int_{\Sigma_n }[ g_{v(t_n)}^{t_0} , \., g_{v(t_1)}^{t_0}] d\L^n + O(s^{n+1}). 
\end{equation}
\end{lemma}

\begin{proof} By formula \eqref{pluto3},
\[
 D_0^nG(w^s ) =\sum_{p=1}^n\sum_{\a\in\I_p, |\a|=n}\frac{n!}{\a!p!} d_0^p G( w^s_\alpha).
\]
 If $\alpha\in \I_p$ has  one entry different from 1, then $d_0^p G( w^s_\alpha)= O(s^{n+1})$ by Lemma~\ref{LEOX}.
 The leading term is given by $p=n$ and $\a\in \I_n$ with $\a=(1,\dots,1)$.
 The expansion of this term is given by 
 formula \eqref{Wn} with $h=n$ and this yields formula \eqref{pluto5}.

\end{proof}

\section{Open mapping property for the extended end-point map} 
\label{OMP}


In this section, we study the open mapping property for the extended end-point map at critical points of corank $1$. 
As in Section~\ref{EXPA}, we denote by   $\q  = F_q(u)$   the end-point and we consider the variation map $G=G_{\q }^u$.
The cokernel  $\coker(d_0 G)$ is a subset of the tangent space $T_{\q } M$. We identify $M$ and $T_{\q } M$ with $\R^m$.

	Let $f_1,\dots, f_d \in \mathrm{Vec}(M)$ be smooth vector fields on the manifold $M$ spanning the distribution $\Delta$ and satisfying the H\"ormander condition \eqref{eq:horm}. If $f_1,\dots, f_d$ are declared orthonormal, the length of a horizontal curve $\gamma\in AC( I ;M)$, 	
	$\dot\gamma = f_u(\gamma)$, is   $L(\gamma)=  \| u\|_{L^1(I;\R^d)}$ while its energy is given by the functional  $J:X\to [0,\infty)$
	  \be\label{eq:energy}  J(u) =  \frac12\| u\|^2_{L^2(I;\R^d)}.
	\ee   The minimizers of $J$ coincide with minimizers of the length by standard arguments.
	The extended end-point map is the map $F_J: X\to M\times\R$
	given by $
	 F _J(u) = (F(u), J(u))$.

	\begin{definition}[Regular, singular, strictly singular]\label{defi:stricts}	
		A critical point  $u\in X$ of $ F_J$ is {\em regular} ({\color{black} resp.,} {\em singular}) if there exists a nonzero $(\lambda,\lambda_0)\in 
	\IM(d_u F_J  )^\perp
	\subset T_{F(u)}^*M\times \R$ such that $\lambda_0\ne 0$ (resp., $\lambda_0= 0$). 
	A critical point  $u\in X$ is {\em strictly singular} if for every $(\lambda,\lambda_0)\in 
	\IM(d_u F_J  )^\perp$ we have  $\lambda_0=0$.
	
	\end{definition}

	We  define the extended variation map $G_J (v) = (F(u+v), J(u+v))$. Then, $0\in X$ is a regular, singular, strictly singular critical point of $G_J$ if and only if  $u$ is a regular, singular, strictly singular critical point of $F_J$.

We are interested in strictly singular critical points of $F_J$. In this case,   $ \mathrm{im}(d_u F_J)=\mathrm{im}(d_u F)\oplus \R$, that is, 
$\coker (d_u F_J)$ and $\coker (d_u F)$ are isomorphic and can be identified.
	The differential analysis of the extended map  $F_J$ can be consequently   reduced to the analysis of the end-point map $F$.
	In fact, for any $h\geq 2$ we have
	$
	\mc{D}_u^h F_J = \mc{D}_u^hF\big|_{\ker(d_u F_J)} $,
	where  the kernel $\ker(d_u F_J)=\ker(d_uF)\cap \ker(d_uJ)$ is   finitely complemented in $X$, and the restriction to	$\ker(d_u F_J)$ means $\dom(\D_0^hF_J)=\{v\in\dom(\D_0^hF)\,|\,v_1\in\ker(d_uJ)\}$.
	Similarly, we have   
\be\label{62}
	\mc{D}_0^h G_J = \mc{D}_0^hG\big|_{\ker(d_0 G_J)} ,\qquad h\geq 2,
	\ee
with $\ker(d_0 G_J) $    finitely complemented in $X$, and
\be
 \label{621}
 \dom(\D_0^hG_J)=\{v\in\dom(\D_0^hG)\,|\,v_1\in\ker(d_uJ)\}.
\ee
Finally, $0\in X$ is a critical point for $G_J$ of corank $\ell=1$ if and only if $u$ is a critical point for $G$ of corank $\ell=1$.

Thanks to the previous remarks, we can without loss of generality consider the situation where
  $0$ is a corank-one critical point for $G$. This means that $\coker(d_0 G)$ has dimension $1$.
We fix a nonzero dual vector  $\la\in\coker(d_0 G)^*$ such that
$
 \langle \la, w\rangle =\proj(w)$, $w\in \R^m$, where $\proj$ is the projection onto $\coker(d_0 G)$.
 
For $n\geq 2$ and  $t_0\in[0,1)$, we consider the function $\mc G^n_{t_0}:X \to\R$
\begin{equation} \label{mom}
\mc G^n _{t_0}(v) = \int_{\Sigma_n } \langle \la,[ g_{v(t_n)}^{t_0} , \., g_{v(t_1)}^{t_0}] \rangle d\L^n 
,\quad 
v\in X .
\end{equation}
This is the coefficient of the leading term in the expansion of $D_0^nG(w_{t_0,s} ) $ in \eqref{pluto5} scalarized with $\lambda$, up to the constant $c_n$.
Here and hereafter, vector fields are evaluated at the end-point $\q  $, with notation as in the previous section.
 
 For a multi-index $\a\in\I_{n,d}$ let us introduce the short notation
\begin{equation}\label{plof}
  [g_{\a}^{t_0}]= [g_{\a_n }^{t_0}, \.,g_{\a_1}^{t_0}],
 \end{equation}
 where the entries $\a_1,\dots,\a_n$ appear in the bracket with reversed order, and 
 \begin{equation}\label{ivo}
I^\a (v) =\int_{\Sigma_n}v^{\a_n}(t_n)\.v^{\a_1}(t_1)d\L^n.
\end{equation}
 Then formula \eqref{mom}
 reads
 \begin{equation}\label{mom1}
\mc G^n _{t_0}(v ) = \sum_{\a\in\I _{n,d}} \langle \lambda, [g_{\a}^{t_0}]\rangle I^\a(v) .
\end{equation}
 
{\color{black} We remind that the space $V_n$ is defined as
$V_n=\bigcap_{i=1}^n U_i,$
where 
\[
U_n =\left\{ v\in X \mid \int_{\Sigma_n}v(t_n)\otimes \dots\otimes v(t_1) d\L^n =0 \right\},
\]
and $\Sigma_n$ is the standard $n$-dimensional simplex (see equations \eqref{defi:simplexts}-\eqref{bip} in Section 4).
} 
 
\begin{lemma}\label{lix}
 If $v\in V_{n-1}$ then $v^\flat \in V_{n-1}$ and $\mc G^n _{t_0}(v^\flat ) =(-1)^{n-1} \mc G^n _{t_0}(v) $.
\end{lemma}

\begin{proof} We have $v^\flat \in V_{n-1}$ by Corollary~\ref{lemma:finitevs}.  By Lemma~\ref{bemolle} -- here we use the assumption $v\in V_{n-1}$, -- the integrals $I^\a(v)$ can be transformed in the following way:
\[
\begin{split}
I^\a (v) & = \int_0^1 v^{\a_n}(t_n)\Big( \int_{\Sigma_{n-1}(t_n,1-t_n)  }v^{\a_{n-1}}(t_{n-1})\.v^{\a_1}(t_1)d\L^{n-1}\Big) dt_n
\\
&
=(-1)^{n-1} \int_0^1 v^{\a_n}(t_n)\Big( \int_{\Sigma_{n-1}^\flat(0,t_n)  }v^{\a_{n-1}}(t_{n-1})\.v^{\a_1}(t_1)d\L^{n-1}\Big) dt_n
\\
&
=(-1)^{n-1}  \int_{\Sigma_n  ^\flat }v^{\a_n}(t_n)\.v^{\a_1}(t_1)d\L^n
\\
&
=(-1)^{n-1}  I^\a(v^\flat).
\end{split}
\]
The last identity follows by the change of variable $t_i= 1-s_i$.
This proves our claim $\mc G^n _{t_0}(v^\flat ) =(-1)^{n-1} \mc G^n _{t_0}(v) $.

 \end{proof}

In the next step, we show that if $\mc G^n _{t_0}$ is positive and additive on a suitable subspace of $V_{n-1}$
then the extended map $G_J$ is open at zero.

{\color{black}

\begin{theorem}
 \label{endopen} 
 Let $0$ be a strictly singular critical point of $G_J$
 with corank $1$ and with  {\color{black} $\dom(\mc D _0^nG_J)$ of finite codimension $h\in\N$.}
 Assume that there exist $t_0\in[0,1)$ and $v_1,\dots, v_k \in V_{n-1}$ such that:
 \begin{itemize}
 \item[i)]  
 $k= h+1 $ when $n$ is even and {\color{black} $k = 2( h+1)$} when $n$ is odd;
   \item[ii)]  $\mc G^n _{t_0}(v_i)=1$ for $i=1,\dots,k$;
  \item[iii)]  $v_1, \dots, v_k$ span a vector space   $Y\subset V_{n-1}$ of dimension $k$;
    \item[iv)] $\mc G^n _{t_0}$ is additive on $ v_1,\dots, v_k$, in the sense that
    \[
\mc G^n _{t_0}\Big(\sum_{i=1}^k \tau _i v_i\Big) =\sum_{i=1}^k    
\mc G^n _{t_0}( \tau _i v_i)
\]
for any $\tau_1,\ldots,\tau_k\in\R$. 
 \end{itemize}
 
 Then the extended map
   $G_J$ is open at $0$.  
   \end{theorem}

\begin{proof}  We denote by $W \subset X$ a $h$-codimensional linear set of elements $w_1 \in X$ such that there exist $w = (w_1,\ldots,w_{n-1})\in\dom (\mc D^n_0 G)$. 
 For $s>0$,  let $L_s:  X= L^1( I ;\R^d)\to X_s { \,\color{black}:=\, } L^1([t_0,t_0+s];\R^d)$ be the linear isomorphism 
$L_s(v) = v_{t_0,s}$ and define
\[
   W_s = L_s^{-1} (W \cap X_s) .
\]
Since   $W\cap X_s$
has codimension at most $h$ in $X_s$, then $L_s^{-1} (W\cap X_s)\subset X$ has codimension at most $h$ 
and thus
\begin{align}
\label{cicop}
&\dim(Y\cap W_s  )\geq k-h= 1& \textrm{when $n$ is even,}
\\
\label{cicop2}
& {\color{black}
\dim(Y\cap W_s )\geq k-h= k/2+ 1}& \textrm{when $n$ is odd.} 
\end{align}

We discuss the case when $n$ is even. By iv), $n$-homogeneity and ii),
for $\tau \in \R^k$, $\tau \neq 0$, we have 
\be\label{tol}
\mc G^n _{t_0}\Big(\sum_{i=1}^k \tau _i v_i\Big) =\sum_{i=1}^k \tau _i^n
 >0.
\ee
Thus,  the function $\mc G^n _{t_0}$ attains a positive minimum on the sphere $K = \{ v\in Y : \| v\|_X=1\}$: there exists $\delta>0$ such that
\begin{equation}
\label{dillo}
\mc G^n _{t_0}(v) \geq \delta>0\quad \textrm{for all }v\in K.
\end{equation}

By \eqref{cicop}, for any $s>0$ there exists $  v^s  \in K$ such that 
$v^s _{t_0,s} =L_s(v^s)  \in  W$. 
Then there exist $w^s_2,\ldots,w_{n-2}^s\in X$ such that $ w^s = (v^s _{t_0,s},w_2^s,\ldots,w_{n-1}^s)\in  \dom(\mc D _0^n G_J)$.
By \eqref{62} and by formula  \eqref{pluto5}
\[
  \mc D _0^n G_J(w^s)=  \mc D _0^n G(w^s) = s^n c_n  \mc G^n_{t_0} (v^s) +O(s^{n+1}),
\]
where $|O(s^{n+1})|\leq C_1 s^{s+1}$ for a constant $C_1>0$ independent of $v^s$ with  $\|v^s\|_X\leq 1$.
Choosing $0<s<\frac 12 \delta/ C_1 c_n $, from \eqref{dillo} we deduce that
\[
\mc D _0^n G_J (w^s)  \geq  s^n \big (c_n   \mc G^n_{t_0} (v^s) - C_1 s \big )\geq s^n  \frac{\delta}{2}>0.
\]

 By Lemma~\ref{lix}, for $n$ even we have    $\mc G^n _{t_0}(v^\flat )=-\mc G^n  _{t_0}(v)$.
 Repeating the above argument starting from $v_1^\flat,\dots, v_k^\flat$,
 we conclude that for all $s>0$ small enough there exists 
  $w^{\flat,s} \in   \dom(\mc D _0^n G_J)$ such that 
$ \mc D _0^n G_J(w^{\flat,s} ) <0$.
 By Corollary~\ref{REGONO} part i), we conclude that $G_J$ is open at $0$.

 {\color{black}

 We pass to the case when $n$ is odd. For small $s>0$ let
  \[
 m(s) =    \sup _{v\in K \cap W_s} |\mc G^n_{t_0}(v) |.
 \]
 We claim that there exists $\delta>0$ and an infinitesimal sequence of $s_j>0$ such that $m(s_j)\geq \delta$ for all $j\in\N$.
 If the claim is true, the proof can be  concluded as in the even case.

 By contradiction, assume that $m(s)\to 0 $ as $s\to 0^+$. 
 By compactness, there exists an infinitesimal sequence of $s_j $ such that $Y\cap W_{s_j}$ is converging
 to a subspace of $Y$ that, by \eqref{cicop2}, has 
 dimension at least  $k/2+1$. Since $m(s)\to 0$, we have   $ \mc G^n_{t_0}=0$ on this subspace.
 Identifying $Y$ with $\R^k$,  by \eqref{tol}
 this means  that there exists 
 a $(k/2+1)$-dimensional subspace of $(\tau_1,\ldots,\tau_k)\in\R^k$ such that
 \be \label{PV}
 \tau_1^n+\ldots+\tau_k^n=0.
\ee
 But this is not possible because the maximal dimension of linear spaces contained in the homogeneous variety defined by \eqref{PV} is $k/2$.
 
 }
   
\end{proof}

}

In fact, in order $G_J$ to be open it is sufficient that $\mc G^n _{t_0}$ is  positive at one  element of $V_{n-1}$.

\begin{theorem}
 \label{endopen2} 
  
 Let $0$ be a strictly singular critical point of $G_J$ with corank $1$ and assume that {\color{black}  $\dom(\mc D _0^n G)$,   $n\geq2$, has finite codimension.}
If  there exist $t_0\in[0,1)$ and $v\in V_{n-1}$ such that  $\mc G^n _{t_0}(v)\neq 0$, then  
   $G_J$ is open at $0$.

   \end{theorem}

\begin{proof} For any $k\in \N$, we apply iteratively Proposition~\ref{ITER} to find $2^k$ functions $v_1,\dots, v_{2^k}$ with mutually disjoint support, spanning a linear space in $V_{n-1}$  and such that
\[
\mc G^n _{t_0}\Big( \sum_{i=1} ^{2^k}  v_i\Big) = \sum_{i=1} ^{2^k} 
\mc G^n _{t_0}( v_i) ,
\]
and $\mc G^n _{t_0}( v_i) =\frac{1}{2^{kn} } \mc G^n _{t_0}( v)$. The claim follows from Theorem~\ref{endopen}.
\end{proof}

If $G_J$ is not open at $0$ we have  $\mc G^n _{t_0}(v) =  0$ for all 
$t_0\in[0,1)$ and $v\in V_{n-1}$.
Even though $V_{n-1}$ is not a linear space, we polarize the map   $T= \mc G^n _{t_0}$. 


The polarization of $T: X\to \R$ is the multilinear map $\T: X^n\to\R$ defined in one of the two equivalent ways
\begin{equation}\label{PO}
\begin{split}
\T(v_1,\dots,v_n) & = \frac{\partial ^n}{\partial t_1 \dots \partial t_n} 
T\Big(\sum_{h=1}^n t_h v_h\Big)\bigg| _{t_1=\dots=t_n=0}
\\
&
=\frac{1}{n!}\sum_{\s\in S_n}\sum_{\a\in\I_{n,d} } \langle \lambda,  [g^{t_0}_{\a}]\rangle
 \int_{\Sigma_n}v^{\a_n}_{\s_n}(t_n)\.v^{\a_1}_{\s_1}(t_1)d\L^n. 
\end{split}
\end{equation}

If $Y\subset X $ is  a linear subspace such that $T=0$ on $Y$ then  $ \T=0$ on $Y^n$. 
This follows easily from the differential definition of   polarization.
Linear spaces $Y\subset V_{n-1}$ where $T=0$ can be obtained with the following construction.

\renewcommand{\j}{s}

Fix $w_1,\dots,w_n \in X$, with coordinates $w_i=(w_i^1,\.,w_i^d)$,
and fix a selection function  $\j:\{1,\.,n\}\to\{1,\.,d\}$, $\j(i)= s_i$.
This  function will be used to select (with multiplicity) which  vector-fields from $f_1,\dots,f_d$ appear in the bracket $[g_\a^{t_0}]$.

We can define $v_1,\dots,v_n\in X$ setting, for $ i=1,\dots,n$,
 \begin{equation}\label{vi}
  v_i=(0,\.,0,u^i ,0,\.,0),\quad  \textrm{with } u^i = w_i^{\j_i},
 \end{equation}
where $u^i $ is the $i$th coordinate. Then we define    $u\in X$ as $u = (u^1,\dots, u^n)$.
The function $u$ depends on the selection function $\j$, but we do not keep track of this dependence in our notation.
As in \eqref{ivo},  for a permutation $\s\in S_n$ we let  
\[
  I^\sigma(u) = \int_{\Sigma_n } u^ {\s_n}  (t_n)\cdots 
 u^ {\s_1} (t_1)d\L^n.
\]

When $u\in V_{n-1}$, where now $V_{n-1}$ is defined as in \eqref{bip}
and \eqref{eq:workingspacen} but with $d=n$, the polarization $\T$ takes the following form.

\begin{lemma} 
 \label{lemma54} For any selection function $\j$,  if   $u \in V_{n-1}$       then $v_1,\dots, v_n$ span a linear subspace of $V_{n-1}$ and
  \begin{equation}\label{pollo}
\begin{split}
 \T (v_1,\.,v_n) 
 &=\frac{1}{n!}\sum_{\s\in S_n} \langle \lambda, [g^{t_0}_{\j\s } ]\rangle  I^\sigma(u).
 \end{split}
\end{equation}
\end{lemma}

\begin{proof} Given $\s\in S_n$ and $\a \in \I_{n,d}$, by the structure \eqref{vi} of $v_1,\dots,v_n$, there holds
\[
 \int_{\Sigma_n}v^{\a_n}_{\s_n}(t_n)\.v^{\a_1}_{\s_1}(t_1)d\L^n=0,
\]
as soon as there exists $i$ such that $\a_i \neq \j (\s_i)$. For the surviving terms  it must be $\a = \j\s$ and in this case
\[
 \int_{\Sigma_n}v^{\a_n}_{\s_n}(t_n)\.v^{\a_1}_{\s_1}(t_1)d\L^n= 
 \int_{\Sigma_n}w^{\j_{\s_n}}_{\s_n}(t_n)\.w^{\j_{\s_1}}_{\s_1}(t_1)d\L^n= I^\s(u).
\]
The claim follows from the combinatorial definition of polarization in \eqref{PO}.

\end{proof}

 \section{Generalized Jacobi identities and integrals on simplexes}\label{GJI}

In this section we fix a selection function  $\j:\{1,\.,n\}\to\{1,\.,d\}$ and $u=(u^1,\dots,u^n)$, as in   \eqref{vi}.
For varying $\s\in S_n$, the brackets $[g_\s^{t_0}]=  [g_ {\s _n}^{t_0},\.,g_{\s _1} ^{t_0} ]$ satisfy   several linear relations.
Using    generalized Jacobi identities, we  clean up   formula \eqref{pollo} getting rid of these relations.
{\color{black} Our goal  is to prove Theorem \ref{lisk1} below. We denote by
\[
  S_n^1:=  \{\s\in S_n \mid \s_1=1\},
 \]
 the set of permutations $\s\in S_n$ fixing $1$.}

\begin{theorem} \label{lisk1}
For any selection function $\j$ and  for any $v_1,\dots,v_n$ as in \eqref{vi}, we have the identity
\begin{equation}
 \label{mainpolfin2}
 \T (v_1,\.,v_n)=\frac{1}{(n-1) !}\sum_{\s\in S_n^1 }
 \langle \lambda, [g^{t_0}_{\j\s } ]\rangle  I^{\sigma}(u) .
\end{equation}
\end{theorem}

 The fact  that in \eqref{mainpolfin2} the sum is restricted to  permutations fixing $1$  will be important in the next section.
The proof relies upon the generalized Jacobi identities proved in \cite{BL88}.
For $n,j\in\N$ with $1\leq j\leq n$, let us consider the sets of permutations
\begin{equation}\label{Xnj}
  X_{nj} =\{\xi\in S_n\,|\,\xi_1>\xi_2>\.>\xi_ j=1 \textrm{ and }\, \xi_ j<\xi_{j+1}<\.<\xi_ n\},
\end{equation}
and
\begin{equation}\label{Xn}
X_n = \bigcup_{j=1}^n X_{nj}.
\end{equation}
 The set $X_{n1}$ contains only the identity permutation, while  $X_{nn}$ contains only the order reversing permutation.
 We are denoting  elements of $X_n$  by $\xi$, while in \cite{BL88} they are denoted by $\pi$.

Let $g_1,\dots,g_n$ be elements of a Lie algebra. The action of a permutation $\s\in S_n$ on the iterated bracket
$[g_n,\dots,g_1] = [g_n,[\dots,[g_{2}, g_1]\dots]]$ is denoted by
\[
\s [g_n,\dots,g_1] = [g_{\s _n},\dots,g_{\s _1}] .
\]
The selection function $\j$ acts similarly, $ \j [g_n,\dots,g_1] = [g_{\j _n},\dots,g_{\j _1}] $, and so in the notation used above we have
$[g^{t_0}_{\j\s } ] = \j [g^{t_0}_{\s } ]$.

The generalized Jacobi identities of order $n$ that we need  are
described in the next theorem.

\begin{theorem}
For any Lie elements $g_1,\dots,g_n$ and for any permutation $\s\in S_n$ such that $\s_1\neq 1$,
\begin{equation} \label{povo}
 \Big(\s+\sum_{\xi\in X_n,\,\s\xi(1)=1}(-1)^{\xi^\1(1)}\s\xi\Big)[g_n,\dots,g_1]=0 ,
\end{equation}
where $X_n$ is the set of permutations introduced in \eqref{Xn}.
\end{theorem}

\begin{proof} The proof of formula \eqref{povo}
is contained in \cite{BL88} on pages 117 and 119.
\end{proof}

\begin{lemma} \label{lisk}
{\color{black} For any selection function $s$ and} for any $v_1,\dots,v_n$ as in \eqref{vi}, we have the identity
\begin{equation}
 \label{mainpolfin}
 \T (v_1,\.,v_n)=\frac{1}{n!}\sum_{\s\in S_n^1 }
 \langle \lambda, [g^{t_0}_{\j\s } ]\rangle \sum_{\xi\in X_n} c_\xi I^{\sigma \xi^\1}(u) ,
\end{equation}
where $c_\xi= (-1)^{1+\xi^\1(1)}$.
\end{lemma}

 \begin{proof}
\renewcommand{\r}{\varrho}
Starting from  \eqref{pollo} and using  \eqref{povo}, we get
\begin{equation}
 \label{mainpol3}
 \begin{split}
 n!\,\T (v_1,\.,v_n)& =\sum_{\s\in S_n^1} \langle \lambda ,[ g_{\j\s}^{t_0} ]\rangle  I^\s(u)+
 \sum_{\s\in S_n,\,\s_1\neq1} \langle \lambda , [g_{\j\s}^{t_0}] \rangle  I^\s(u) \\
 &
 =\sum_{\s\in S_n^1}  \langle \lambda , [g_{\j\s}^{t_0}] \rangle  I^\s(u)
 +\sum_{\substack{ \xi\in X_n,\, \s\in S_n\\ \s_1\neq1,\,\s \xi(1)=1}} c_\xi \langle \lambda , [g_{\j\s\xi}^{t_0}] \rangle  I^\s(u)
 \\
 &
 =\sum_{\s\in S_n^1}  \langle \lambda , [g_{\j\s}^{t_0} ]\rangle  \Big(I^\s(u)+\sum_{{ \xi\in X_n,\, \s\xi^\1(1)\neq1} } c_\xi I^{\s\xi^\1}(u)\Big)
 \\
 &=\sum_{\s\in S_n^1}  \langle \lambda , [g_{\j\s}^{t_0}] \rangle   \sum_{\xi\in X_n} c_\xi I^{\s \xi^\1}(u) .
 \end{split}
\end{equation}
In the last line, we used the fact that, when $\s_1=1$, we have $\s\xi^{-1} (1) =1$ if and only if $\xi$ is the identity.

\end{proof}

  A permutation $\sigma\in S_n$ acts on the integrals $I^{\xi^{-1}}(u)$ as $\s I^{\xi^{-1}}(u)= I^{\s\xi^{-1}}(u)$. So,
 the sum over $\xi\in X_n$ appearing in \eqref{mainpolfin} reads
 \be \label{Intact}
  \sum_{\xi\in X_n} c_\xi I^{\sigma \xi^\1}(u) = \s \Big ( \sum_{\xi\in X_n} c_\xi I^{\xi^\1}(u)\Big),
 \ee
  where the action is extended linearly.
Our next task is to compute the sum in the round brackets.

  A permutation $\sigma\in S_n$  acts on the simplexes $\Sigma_n(t,s)$, with $0\leq t <t+s\leq 1$, as
 \be \label{simact}
  \s\Sigma_n(t,s)
  =\big\{ (t_1,\ldots,t_n)\in \R^n \mid
  t< t_{\s_n}<\ldots <t_{\s_1}<t+s\big\}.
 \ee
 In particular,
if  $\bar\s\in S_n$ is the reversing order permutation,  $\bar\s(i)=n-i+1$, then
$\Sigma_n^\flat(t,s)=\bar \sigma \Sigma_n  (t,s)$.
We   also let
$\Sigma_n^\sigma (t,s) =  \s\Sigma_n(t,s)$ and
$ \Sigma_n^\sigma =\Sigma_n^\sigma  (0,1)$.

Finally, for $k=1,\ldots,n$ we let
\[
 \begin{split}
 I_k^\flat(u) &=  \int_{\Sigma_{k}^\flat }  u^ {1} (s_1)\dots u^k(s_k)  d\L^k   ,
 \\
 I_{n-k}  (u) &=
   \int_{\Sigma_{n-k}  }   u^{k+1} (s_{k+1} )\dots u^{n}(s_n)   d\L^{n-k}.
   \end{split}
\]

 \begin{lemma} \label{oggi}
  For any $j=2,\dots, n$ we have the identity
  \begin{equation}\label{fan}
   \sum_{\xi\in X_{nj}} I^{ \xi^\1}(u)=
    I_{j-1}^\flat  (u)I_{n-j+1}  (u)
   -
    \sum_{\xi\in X_{n,j-1}} I^{ \xi^\1}(u).
 \end{equation}
 \end{lemma}

\begin{proof}

Fix a permutation $\xi\in X_{nj}$, so that $\xi_ j =1$. In the integral
 $I^{\xi^{-1}}(u)$ we perform the   change of variable $t_{\xi_ k}=s_k$. The integration domain
$\Sigma_n$ is transformed into the new  domain    $ \Sigma_n^{\xi^\1} = \{0<s_{ \xi ^{-1}_ n}<\dots <s_{ \xi ^{-1} _1}<1\}$:
 \begin{equation} \label{buf}
 \begin{split}  I^{ \xi^\1} (u) & =  \int_{\Sigma_n }  u^ {\xi^{-1} _n} (t_n)\dots u^ {\xi^{-1} _1} (t_1) d\L^n
 \\
 &
 =  \int_{\Sigma_n }  u^{n} (t_{\xi_ n})\dots u^{1} (t_{\xi _1}) d\L^n
  \\
 & =
  \int_{\Sigma_n^{ \xi^{-1}} }    u^{n} (s_n)\dots  u^{1} (s_1) d\L^n.
 \end{split}
 \end{equation}

We denote by $ \widehat s_j$ the variables $(s_1,\dots, s_{j-1}, s_{j+1},\dots,s_n)$.
Since   $s_{ \xi ^{-1} 1}=s_j$, the set $   \Sigma_n ^{\xi^{-1} }$ is
 \[
   \Sigma_n ^{\xi^{-1} }= I  \times\Sigma_{n-1;j} ^{ \xi^\1} (0,s_j),
\]
where $s_j\in I $ and
\[
\Sigma_{n-1;j} ^{ \xi^\1} (0,s_j)=\big\{ \widehat  s_j\in\R^{n-1} \mid0<s_{\xi^\1_n}<\.<s_{\xi^\1_2}<s_j \big\}.
\]
Since $\xi\in X_{nj}$, here we have $s_n<\.<s_{j+1}<s_j$ and $s_1<\.<s_{j-1}<s_j$. For varying  $\xi\in X_{nj}$, we obtain all the shuffles of  $s_1<\.<s_{j-1}$ into $s_n<\.<s_{j+1}$ and thus we have
\[
\bigcup_{\xi \in X_{nj} }
\Sigma_{n-1;j} ^{ \xi^\1} (0,s_j)=  A_{j-1}(s_j)\times B_{n-j}(s_j),
\]
 where $A_{j-1}(s_j)=
\Sigma_{j-1}^\flat (0,s_j)$ and $  B_{n-j}(s_j)=
\Sigma_{n-j}  (0,s_j)$.

 Summing \eqref{buf} over  $\xi\in X_{nj}$  we   get
  \begin{equation}\label{forx}
 \begin{split}
 \sum_{\xi\in X_{nj}} I^{ \xi^\1}(u) &
 =
 \sum_{\xi\in X_{nj}}  \int_0^1 \Big(   \int_{\Sigma_{n-1} ^{ \xi^{-1}} (0,s_j) }\prod_{k\neq j}   u^{k} (s_k ) d\L^{n-1}(\widehat s_j)   \Big) u^j(s_j) ds_j
 \\
 &
 =   \int_0^1 \Big(   \int_
 {A_{j-1}(s_j)\times B_{n-j}(s_j) }
 \prod_{k\neq j}   u^{k} (s_k ) d\L^{n-1}(\widehat s_j)  \Big) u^j(s_j) ds_j.
 \end{split}
 \end{equation}
 The inner integral is the product of two integrals. Namely, letting
\begin{equation}\label{fg}
\begin{split}
f(s_j) & =  u^j(s_j)   \int_{ B_{n-j}  (s_j)} u^{j+1} (s_{j+1} ) \. u^{n} (s_{n} ) d\L^{n-j},
\\
g(s_j)&=     \int_{
A_{j-1} (s_j) }  u^ {j-1} (s_{j-1} )\dots u^ 1 (s_1) d\L^{j-1} ,
\end{split}
 \end{equation}
 formula \eqref{forx}
 becomes
  \begin{equation}\label{forx1}
 \begin{split}
 \sum_{\xi\in X_{nj} }I^{ \xi^\1}(u)   =   \int_0^1  f(s_j)   g(s_j ) ds_j .
 \end{split}
 \end{equation}

 A primitive for $f$ is the function    $h(s_j) = \int_0^ {s_j}  f(\sigma) d\sigma $, and an integration by parts gives
 \[
 \begin{split}
  \int_0^1  f(s_j)   g(s_j ) ds_j& = h(1) g(1)  -\int_0^1 h(s_j) g'(s_j) ds_j,
  \end{split}
\]
where the boundary term is easily computed:
\[
\begin{split}
  h(1)  & =   \int_{B_{n-j+1} (1) }  u ^{j} (s_{j})\dots u^ n (s_n) d\L^{n-j+1}  = I_{n-j+1}(u) ,
\\
  g(1) &=  \int_{A_{j-1}(1)  }  u^{j-1} (s_{j-1} )\dots u^{1} (s_{1} ) d\L^{j-1}    = I^\flat_{j-1}(u).
\end{split}
\]
In order to compute the integral, notice that
\[
   g'(s_j)=   u^{j-1} (s_{j})  \int_{A_{j-2}(s_j) }  u^ {j-2} (s_{j-2} )\dots u^1  (s_{1} ) d\L^{j-2},
\]
and thus, by \eqref{forx} but for $j-1$, we have
\[
\begin{split}
\int_0^1 h(s_j) g'(s_j) ds_j& =\int_0^1    \Big(   \int_{A_{j-2}(\sigma) \times B_{n-j+1} (\s) } \prod_{k\neq j-1}   u^{k} (s_k ) d\L^{n-1}(\widehat s_{j-1} ) \Big)  u ^{j-1} (\sigma ) d\sigma
\\
&=
 \sum_{\xi\in X_{n,j-1}}  I^{ \xi^\1}(u) .
\end{split}
\]

\end{proof}

\begin{corollary}
	\label{corfin7}
For any  $u\in V_{n-1}$ there holds
 \[
  \sum_{\xi\in X_n } c_\xi I^{\xi^\1}(u) = 
  n     \int_{\Sigma_{n}   } u^1(t_1) \dots u^n(t_n)     d\L^n.
   \]

\end{corollary}

\begin{proof} When  $u\in V_{n-1}$ we have
$
    I^\flat_{j-1}(u)=I_{n-j+1}(u)=0$ for $j=2,\dots,n$. Taking into account the constants  $c_\xi=(-1)^{1+\xi^\1_1}$, formula \eqref{fan} reads
    \begin{align*}
    \sum_{\xi\in X_{nj}}c_\xi  I^{ \xi^\1}(u)&=      \sum_{\xi\in X_{n,j-1}} c_\xi I^{ \xi^\1}(u).
    \end{align*}
    Applying iteratively this identity, we obtain
    \begin{align*}
     \sum_{\xi\in X_{n}}c_\xi  I^{ \xi^\1}(u)&= \sum_{\xi\in X_{nn}} c_\xi I^{ \xi^\1}(u)
     +\sum_{\xi\in X_{n,n-1}} c_\xi I^{ \xi^\1}(u)+\.+\sum_{\xi\in X_{n1}} c_\xi I^{ \xi^\1}(u)
     \\
     &=2\sum_{\xi\in X_{n,n-1}} c_\xi I^{ \xi^\1}(u)
     +\sum_{\xi\in X_{n,n-2}} c_\xi I^{ \xi^\1}(u)+\.+\sum_{\xi\in X_{n1}} c_\xi I^{ \xi^\1}(u)
     \\
     &=3\sum_{\xi\in X_{n,n-2}} c_\xi I^{ \xi^\1}(u)
     +\sum_{\xi\in X_{n,n-3}} c_\xi I^{ \xi^\1}(u)+\.+\sum_{\xi\in X_{n1}} c_\xi I^{ \xi^\1}(u)
     \\
     &=\.
     \\
     &=n\sum_{\xi\in X_{n1}} c_\xi I^{ \xi^\1}(u).
    \end{align*}
    Since $X_{n1}$ contains only the identity permutation, the claim follows.
\end{proof}

{\color{black}
  With this corollary, we can conclude the proof of Theorem \ref{lisk1}. By \eqref{Intact}, identity \eqref{mainpolfin} of Lemma \ref{lisk} reads
  \begin{equation}
  	\label{mainpolfin3}
  	\T (v_1,\.,v_n)=\frac{1}{n!}\sum_{\s\in S_n^1 }
  	\langle \lambda, [g^{t_0}_{\j\s } ]\rangle \s\left(\sum_{\xi\in X_n} c_\xi I^{ \xi^\1}(u)\right).
  \end{equation} 
  Then, applying Corollary \ref{corfin7} to \eqref{mainpolfin3} we get
  \begin{equation}
  	\begin{split}
  			\T (v_1,\.,v_n)&=\frac{1}{n!}\sum_{\s\in S_n^1 }
  		\langle \lambda, [g^{t_0}_{\j\s } ]\rangle \s(nI(u))\\
  		&=\frac{1}{(n-1)!}\sum_{\s\in S_n^1 }
  		\langle \lambda, [g^{t_0}_{\j\s } ]\rangle I^\s(u),
  	\end{split}
  \end{equation}
 completing the proof.
}

\section{Non-singularity via trigonometric functions}\label{NSVTF}
 
We start the study of   equation $\T(v_1,\dots,v_n)=0$ for the polarization map $\T$   in \eqref{mainpolfin2}. We will work with functions $v_i$ as in \eqref{vi} of
trigonometric-type.
 
  For each permutation fixing 1,  $\s\in S_n^1$, we introduce a real unknow $x_\s$. There are   $(n-1)!=\mathrm{Card} (S_n^1)$ unknowns.
We are interested
	 in the linear system of equations
\begin{equation}
\label{PUF9}
 \sum_{\s\in S_n^1}   I^{\sigma}(u_\tau)  x_\s=0,\quad \tau \in S_n^1,
\end{equation}
where 
$I^{\sigma}(u_\tau)$ are regarded as coefficients depending on $u_\tau \in V_{n-1}$.  In this section, we prove the following preparatory result.

\begin{theorem}\label{TRIX}
There exist $u_\tau \in V_{n-1}$,   $\tau\in S_n^1$, such that $\det( I^{\sigma}(u_\tau))_{\s,\tau\in S_n^1}\neq 0$.
\end{theorem}

With a choice of coefficients as in Theorem~\ref{TRIX}, the linear system \eqref{PUF9} has only the zero solution, implying $x_\s=0$ for all $\s\in S_n^1$. This fact will be used in Section~\ref{GOH}.

 For   $z=a+i b\in\C$ and $k\in\N$ we let  
 \[
        w_{z;k} (t) = a \cos (2k\pi  t) + b\sin (2k\pi t),\quad t\in I .
\]
We call $w_{z;k}$ a $w$-type function of parameters $z$ and $k$, and we call $k$ frequence of $w_{z;k}$. 
We need exact computations for iterated integrals on $n$-simplexes of $w$-type functions. 
In particular, we are interested in the case when every  linear combination  with coefficients $\pm1$ of at most $n-1$ frequences out of a set of $n$ frequences   is not  zero, see   \eqref{preliminare} below. This condition will ensure   assumption $u\in V_{n-1}$ in Lemma~\ref{lemma54}.

Any $w$-type function satisfies the integration  formula
\begin{equation}
    \label{integration}
     \int_t^1 w_{z;k}(s) ds = \frac{1}{2k\pi} \big (w_{iz;k}(1)-w_{iz;k}(t)\big) , \quad k\neq0,
     \end{equation}
     and a pair of $w$-type functions satisfies  the   multiplication formula (Werner's identities)
     \begin{equation}
    \label{multiplication}
     w_{z;k}  w_{\zeta;h}
     =\frac 12 \big( w_{z\zeta;k+h} + w_{\bar z \zeta; h-k}\big).
\end{equation}

For $h\in\N$, we let 
 $\J_h=\{1,\.,h\}$ and 
\[
  \A_h=\big\{\a:\J_h\to\{1,-1\}\mid \a_1=1\big\}.
\]
Here, we are denoting    $\a(j)=\a_j$ and, with a slight abuse of notation, we identify $\a\in\A_h$ with   $\a=(\a_1,\.,\a_h)\in\{1,-1\}^h$.
Letting   $\z_h=(z_1,\.,z_h)\in\C^h$, 
 for each  $\a\in\A_h$ we define the multiplicative function   $p_\a:\C^{h}\to\C$ 
\[
 p_\a(\z_h)=\prod_{\ell\in\J_h, \a_\ell=1}z_\ell \prod_{j\in\J_h, \a_j=-1}\bar z_j.
\]
Also, letting  $\k_h=(k_1,\.,k_h)\in\N^h$, we define the additive function $s_\a:\N^h\to\N$
\[
 s_\a(\k_h)=\sum_{j=1}^h\a_jk_j.
\]
Notice that, since $\a_1=1$, $\bar z_1$ never appears  in $p_\a(\z_h)$ and $k_1$ has always positive sign in $s_\a(\k_h)$.

Finally,  for $\ell,h\in\N$ with $\ell\leq h$ we let $\B_\ell^h =\{\b:\J_\ell\to\J_h \mid \b \textrm{ injective}\}$ and for $\k _h =(k_1,\dots,k_h)\in \N^h$
and $\b\in \B_\ell^h$   we set $\k_h^\b = (k_{\b_1},\dots, k_{\b_\ell})\in\N^\ell$.
Here, we are using the math-roman font for vectors and   italics  for   coordinates.

\begin{theorem}
 \label{bugsbunny}
 
 Let $\k_n=(k_1,\dots,k_n) \in \N^n$, $n\in\N$,  be a vector of frequences such that 
  \be
  \label{preliminare}
  s_\a(\k_n^\beta)  \neq0 \quad \textrm{for all } \a\in\A_h\textrm{ and }\, \b\in\B_h^n,\,\,\ 1\leq  h\leq n-1.
 \ee
 Then   for any   $1\leq h \leq n-1$, for all $\z_h  =(z_1,\dots,z_h) \in\C^h$ and for all $t\in I $ we have 
  \be\label{85}
  \int_{\Sigma_h(t,1-t)}w_{z_h;k_h}(t_h) \.w_{z_1;k_1}(t_1) d\L^h=g_{\z_h;\k_h}^h(t)-\sum_{\a\in\A_h}c_\a^h (\k_h)  w_{p_\a(i\z_h);s_\a(\k_h)} (t) ,
 \ee
 where $c_\a^h (\k_h) \neq 0$ and the function $g_{\z_h;\k_h}^h$ satisfies 
  \be\label{857}
  g_{\z_h;\k_h}^h(0)=\sum_{\a\in\A_h}c_\a^h (\k_h)  w_{p_\a(i\z_h);s_\a(\k_h)} (0) ,
 \ee
 and 
  \be\label{856}
  \int_{\Sigma_{n-h}}w_{z_n;k_n}(t_n) \.w_{z_{h+1};k_{h+1}}(t_{h+1})  g_{\z_h;\k_h}^h(t_{h+1})  d\L^{n-h}=0.
  \ee
  {\color{black} The constants $c_\a^h(\k_h)$ in \eqref{857} are given by the   formula 
  \be\label{ca}
  c_\a^h(\k_h) =\frac{2}{4^h \pi^h}\prod_{\ell=1}^h\frac{\a_\ell}{s_{(\a_1,\.,\a_\ell)}(\k_\ell)}.
  \ee
}
 
            \end{theorem}

\begin{proof}
 The proof is by induction on $n\in\N$, $n\geq 2$. 
 The formula {\color{black} \eqref{ca}} is well-defined because $s_{(\a_1,\.,\a_\ell)}(\k_\ell)\neq0$
 by assumption \eqref{preliminare}.

 When $n=2$ we only have  $h=1$ and  $\a=1$, so that $c_1^1(\k_1)= 1/2\pi k_1$.
 The integration formula in  \eqref{integration} gives \eqref{85} with     $g_{\z_1;\k_1}^1 = c_1^1(\k_1)w_{iz_1;k_1}(1)$, a constant.
 Identity \eqref{857} is satisfied and also identity  \eqref{856}:
 \[
 \int_{0}^1  w_{z_2;k_2}(t_2) g_{\z_1;\k_1}^1 dt_2 = g_{\z_1;\k_1}^1
 \int_{0}^1  w_{z_2;k_2}(t_2)  dt_2 =0,
 \]
 because $k_2\neq 0$, again by   \eqref{preliminare}.

 Now we assume   the theorem holds for $n-1$ and we prove it for $n$.
In particular, from \eqref{85} with $t=0$ and \eqref{857} we have the   inductive assumption
\be\label{851}
 \int_{\Sigma_{h}  }w_{z_{h} ;k_{h} }(t_{h} ) \.w_{z_1;k_1}(t_1) d\L^h=0, \quad h=1,\dots,n-2 .
 \ee
 
 We distinguish the cases $h=1$ and $2\leq h \leq n-1$. When $h=1$, \eqref{85} is exactly the integration formula 
 \eqref{integration} with 
 \[
  g^1_{z_1,k_1}=\frac{1}{2k_1\pi}w_{z_1;k_1}(1).
 \]
 The 1-periodicity of $w$-type functions also proves \eqref{857}. In order to prove \eqref{856}, we claim that
 \[
  \int_{\Sigma_{n-1}}w_{z_n;k_n}(t_n) \.w_{z_2;k_2}(t_2)  d\L^{n-1}=0.
 \]
 In fact,
 \begin{align*}
  \int_{\Sigma_{n-1}}w_{z_n;k_n}(t_n) \.w_{z_2;k_2}(t_2)  d\L^{n-1}= \;g^1_{z_2,k_2}\int_{\Sigma_{n-2}}w_{z_n;k_n}(t_n) \.w_{z_3;k_3}(t_3)&d\L^{n-2} \\
  - \frac{1}{2k_2\pi}\int_{\Sigma_{n-2}}w_{z_n;k_n}(t_n) \.w_{z_3;k_3}(t_3)w_{z_2;k_2}(t_3)&d\L^{n-2}.
 \end{align*}
 By the multiplication formula \eqref{multiplication}, in the second integral are involved the vectors of frequences $(k_2\pm k_3,k_4,\.,k_n)$. Both of them satisfy  \eqref{preliminare}, by assumption \eqref{preliminare} itself.
 Then both the summands vanish thanks to the inductive assumption \eqref{851}, proving our claim.

 For $2\leq h\leq n-1$, set
 \[
 D_h(t)=
  \int_{\Sigma_h(t,1-t)}w_{z_h;k_h}(t_h) \.w_{z_1;k_1}(t_1) d\L^{h}.
 \]
 When $h\geq2$, we use the inductive assumption \eqref{85} for $h-1$  and  the multiplication formula \eqref{multiplication}  
 to obtain 
 \begin{align*}
 D_h(t) & =\int_t^{1} w_{z_h;k_h}(t_h)
  \Big(\int_{\Sigma_{h-1} (t_h,1-t_h)}w_{z_{h-1} ;k_{h-1}}(t_{h-1}) \.w_{z_1;k_1}(t_1) d\L^{h-1}\Big)dt_h
 \\
  &=\int_t^1w_{z_h;k_h}(t_h) \Big(g^{h-1}_{\z_{h-1};\k_{h-1}}-
  \sum_{\a\in\A_{h-1}}c_\a^{h-1} (\k_{h-1}) w_{p_\a(i\z_{h-1});s_\a(\k_{h-1})}  \Big)dt_h\\
  &=\int_t^1  w_{z_h;k_h}g^{h-1}_{\z_{h-1};\k_{h-1}} dt_h -\frac12
  \sum_{\a\in\A_{h-1}}c_\a^{h-1}(\k_{h-1})  \int_t^1  \big(w_{**}+w_{\dagger\dagger}\big) dt_h,
 \end{align*}
 where  $**=z_hp_\a(i\z_{h-1});s_\a(\k_{h-1})+k_h$ and $\dagger\dagger=\bar z_h p_\a(i\z_{h-1});s_\a(\k_{h-1})-k_h$ satisfy 
  \[
 \begin{split}
  w_{** }& =
  -w_{ip_{(\a,1)}(i\z_h);s_{(\a,1)}(\k_h)},
\\
  w_{\dagger\dagger}&=w_{ip_{(\a,-1)}(i\z_h);s_{(\a,-1)}(\k_h)}.
  \end{split}
 \]
 By the integration formula \eqref{integration},  the function $D_h$ equals  
  \[
 \begin{split}
 D_h& = g^h_{\z_h,\k_h}-\frac{1}{4\pi} 
  \sum_{\a\in\A_{h-1}}c_\a^{h-1}(\k_{h-1} ) \Big(\frac{w_{p_{(\a,1)}(i\z_h);s_{(\a,1)}(\k_h)}}{s_{(\a,1)}(\k_h)}-\frac{w_{p_{(\a,-1)}(i\z_h);s_{(\a,-1)}(\k_h)}}{s_{(\a,-1)}(\k_h)}\Big),
 \end{split}
 \]
 where
 \be\label{ciccio}
g^h_{\z_h,\k_h}(t)=  \int_t^1 w_{z_h;k_h}g^{h-1}_{\z_{h-1};\k_{h-1}}dt_h +c(\z_h,\k_h),
\ee
with $c(\z_h,\k_h)$ a constant that we are going to fix in a moment.
Using  $
  \A_{h}=\{(\a,1)\mid \a\in\A_{h-1}\}\cup\{(\a,-1)\mid \a\in\A_{h-1}\}$, and
  the relations  
 \[
  \frac{1}{4\pi}c_\a^{h-1}(\k_{h-1})\frac{1}{s_{(\a,1)}(\k_h)}=c_{(\a,1)}^h(\k_h)
   \quad \text{and} \quad 
   \frac{1}{4\pi}c_\a^{h-1}(\k_{h-1})\frac{1}{s_{(\a,-1)}(\k_h)}=-c_{(\a,-1)}^h(\k_h),
 \]
 we conclude that 
 \[
 D_h(t) =  g^h_{\z_h,\k_h}(t)-
  \sum_{\a\in\A_{h} }c_\a^h (\k_h) w_{p_\a(i\z_h);s_\a(\k_h)}(t).
 \]
 This proves \eqref{85}.
 
 We are left with the proof of \eqref{857} and \eqref{856}.
 The constant above is 
 \[
 c(\z_h,\k_h)=
  \sum_{\a\in\A_{h} }c_\a^h (\k_h) w_{p_\a(i\z_h);s_\a(\k_h)}(1).
  \]
  By the $1$-periodicity of $t\mapsto g^h_{\z_h,\k_h}(t)$, we have
  \[
g^h_{\z_h,\k_h}(0)= g^h_{\z_h,\k_h}(1) =c(\z_h,\k_h),
  \]
  and this shows \eqref{857}.

  Finally, we check  \eqref{856}. By \eqref{ciccio} it is sufficient to show that
  \be\label{8561}
  \int_{\Sigma_{n-h}}w_{z_n;k_n}(t_n) \.w_{z_{h+1};k_{h+1}}(t_{h+1})  \Big( 
   \int_{t_{h+1}}^1 w_{z_h;k_h}g^{h-1}_{\z_{h-1};\k_{h-1}}dt_h  \Big) 
   d\L^{n-h}=0
  \ee
  and 
  \be\label{8562}c(\z_h,\k_h)
  \int_{\Sigma_{n-h}}w_{z_n;k_n}(t_n) \.w_{z_{h+1};k_{h+1}}(t_{h+1})   d\L^{n-h}=0.
  \ee
Identity \eqref{8562} holds by  \eqref{851} and identity \eqref{8561} holds by the inductive validity of  \eqref{856}.

\end{proof}

 The explicit formula \eqref{ca} for the constants $c_\a^{h}(\k_h)$ will be crucial  in Lemma~\ref{lemma85}.

\begin{corollary}
 \label{bugsbunny2}
 Let $\k_n=(k_1,\dots,k_n) \in \N^n$, $n\in\N$,  be a vector of frequences  satisfying \eqref{preliminare}
 and assume there exists a unique $\bar \a\in\A_n$ of the form $\bar\a=(\a,-1)$ with $\a\in\A_{n-1}$ such that $s_{\bar\a} (\k_n)=0$. Then we have
  \be\label{95}
  \int_{\Sigma_n}w_{z_n;k_n}(t_n) \.w_{z_1;k_1}(t_1) d\L^{n}=-\frac 12 c_\a^{n-1} (\k_{n-1} )  \Re\big( \bar z_n p_\a(i\z_{n-1} )\big)   .
 \ee             \end{corollary}

 \begin{proof}
 Using formulas \eqref{85} and \eqref{856} 
  we obtain
 \[
 \begin{split}
  \int_{\Sigma_n} w_{z_n;k_n} \.w_{z_1;k_1} & d\L^{n} 
  = \int_0^1 w_{z_n;k_n} \Big(  \int_{\Sigma(t_n,1-t_n)}
   w_{z_{n-1};k_{n-1}}
  \.w_{z_1;k_1} d\L^{n-1} \Big) dt_n
  \\
  &
  = \int_0^1 w_{z_n;k_n}\Big(
  g_{\z_{n-1};\k_{n-1} }^{n-1}
  -\sum_{\a\in\A_{n-1} }c_\a^{n-1} (\k_{n-1})  
  w_{p_\a(i\z_{n-1});s_\a(\k_{n-1})} \Big)dt_n
  \\
  &
  =-c_\a^{n-1} (\k_{n-1})    \int_0^1 w_{z_n;k_n}   w_{p_\a(i\z_{n-1});s_\a(\k_{n-1})}  dt_n.
   \\
  \end{split}
 \]
 Now we use the multiplication formula \eqref{multiplication}. Only the one term with a resulting zero frequence contributes to the integral, and we get
 \[
 \begin{split}
  \int_{\Sigma_n}w_{z_n;k_n} \.w_{z_1;k_1} d\L^{n}   =-\frac 12 c_\a^{n-1} (\k_{n-1}) \Re ( \bar z_n p_\a(i\z_{n-1})).
  \end{split}
 \]
 
 \end{proof}

 \begin{remark}
 \label{rem:freq}
 Let $\k_n=(k_1,\dots, k_n)\in\N^n$ be a vector of frequences such that  
\begin{equation}
 \label{frequenze}
 \begin{split}
 k_1=\sum_{j=2}^nk_j\quad \textrm{and}
 \quad  k_j>\sum_{\ell=j+1}^n k_\ell, \quad 2\leq j\leq n-1. 
 \end{split} 
\end{equation}
 Then $\k_n$ stisfies \eqref{preliminare} and there exists a unique $\bar \a= (\a,-1)\in\A_n$ such that $s_{\bar\a}(\k_n)=0$, and namely
  $\bar \a = (1,-1,\dots,-1)$.
 \end{remark}

 \begin{lemma}
 \label{lemma85}
 There exists $\k_n=(k_1,\dots, k_n)\in\N^n$ as in \eqref{frequenze} such that, with  $\bar \a = (1,-1,\dots,-1)$, there holds
 \be
  \label{num1}
 | c_{\bar\a}^{n-1}(\k_{n-1})|>\sum_{\s\in S_n^1,  \s\neq id}  | c_{\bar\a}^{n-1}(k_{\s_1},  \dots, k_{\s _{n-1}})|.
 \ee
 \end{lemma}
 
 \begin{proof} Setting, for   $\ell=3,\dots,n$,  
  \[
  q(k_\ell,\.,k_n)=\prod_{i=\ell}^n\frac{1}{k_\ell+\.+k_n} ,
   \]
by formula \eqref{ca} and by the choice of $k_1$ in \eqref{frequenze}, we obtain
%
  \[
  \begin{split}
 | c_{\bar\a}^{n-1}(\k_{n-1})| & 
 =\frac{2}{4^{n-1}\pi^{n-1} k_1}  q(k_3,\dots,k_n) ,  
 \end{split}
 \]
and so  inequality \eqref{num1} is equivalent to 
 \be
  \label{num2}
  q(k_3,\.,k_n)>\sum_{\s\in S_n^1, \s\neq id} q(k_{\s _3},\.,k_{\s_ n}). 
 \ee
Notice that $k_1$ does not appear in  \eqref{num2}, whereas $k_2$ may appear in the right hand side.

For $i=1,\dots,n$, consider the set of permutations fixing $1,\dots,i$:
 \[
  S_n^i=\{\s\in S_n\mid \s_1=1,\.,\s_ i=i\}.
 \]
 We claim that there exist $k_2,\dots, k_n$ as in \eqref{frequenze}, such that for any $\ell=3,\dots, n$ there holds
  \be\label{num4}
  q(k_{\ell},\.,k_n)>\sum_{i=\ell-2}^{n-2}\sum_{\s\in S_n^i, \s(i+1)\neq i+1}  q(k_{\s _\ell} ,\.,k_{\s _n}).
 \ee
Claim \eqref{num4} for $\ell=3$ is exactly \eqref{num2}.

 We prove \eqref{num4} by induction on $\ell$ starting from $\ell=n$ and descending. When $\ell=n$, the sums in the right hand side of \eqref{num4}
 reduce to the sum on one element, the permutation switching $n$ and $n-1$. So, inequality \eqref{num4} reads in this case
 \[
  \frac{1}{k_n}=q(k_n)>q(k_{n-1})=\frac{1}{k_{n-1}},
 \]
 that holds as soon as  $0<k_n<k_{n-1}$.

 By induction, assume that $k_{\ell}>\dots>k_n$ are already fixed in such a way that \eqref{num4} holds with $\ell+1$ replacing $\ell$. Notice that $k_{\ell-1}$ may and indeed does appear in the right hand side.
 We claim that there exists $k_{\ell-1} > k_{\ell}$ such that \eqref{num4} holds. 
 
 We split the sum in \eqref{num4}  obtaining
 \be
  \label{num5}
  q(k_{\ell},\.,k_n)>\sum_{i=\ell-1}^{n-2}\sum_{\s\in S_n^i, \s(i+1)\neq i+1} q(k_{\s _\ell},\.,k_{\s_ n}) + 
  \sum_{\s\in S_n^{\ell-2}, \s(\ell-1)\neq \ell-1}  q(k_{\s_\ell},\.,k_{\s _n}),
 \ee
 and we consider the quantity
  \begin{align*}
  Q(k_\ell,\dots,k_n) =  q(k_\ell,\.,k_n)-\sum_{i=\ell-1}^{n-2}\sum_{\s\in S_n^i, \s(i+1)\neq i+1}  q(k_{\s_ \ell},\.,k_{\s_ n}) 
 \end{align*}
 
 A permutation $\s\in S_n^i$ with $i\geq \ell-1$ fixes all the $k_i$s with $i\leq  \ell-1$ and then we have
 \[
 q(k_{\s \ell},\.,k_{\s n})=\frac{1}{k_\ell+\.+k_n}  q(k_{\s_{\ell+1}},\.,k_{\s _n}),
 \]
 and thus 
  \begin{align*}
    Q(k_\ell,\dots,k_n)  =\frac{1}{k_\ell+\.+k_n}\Big(q(k_{\ell+1},\.,k_n)-\sum_{i=\ell-1}^{n-2}\sum_{\s\in S_n^i, \s(i+1)\neq i+1}  q(k_{\s_{\ell+1}},\.,k_{\s_ n})\Big).
 \end{align*}
 By our inductive assumption, 
we have  $  Q(k_\ell,\dots,k_n)>0 $. 
 Notice that  $    Q(k_\ell,\dots,k_n) $ does not depend on $k_{\ell-1}$.
 
 Conversely, every summand in the second sum in \eqref{num5}, i.e., every $ q(k_{\s_ \ell},\.,k_{\s _n})$,  depends on $k_{\ell-1}$ 
 and tends to $0$ as $k_{\ell-1}\to\infty$.
 We conclude that   for all large enough  $k_{\ell-1}>k_\ell$ claim \eqref{num5} holds. This ends the proof of   \eqref{num4}.

 \end{proof}

 \begin{proof}
 [Proof of Theorem~\ref{TRIX}]
  We claim that for each $\tau\in S_n^1$ there exists $u_\tau\in V_{n-1}$ such that
 the matrix $(I^\sigma(u_\tau))_{\s,\tau\in S_n^1}$ is strictly diagonally dominant and thus invertible.

 Let $\k_n=(k_1,\dots,k_n)\in  \N^n$ be a vector of frequences as in \eqref{frequenze} and satisfying the claim of Lemma~\ref{lemma85} and
   choose complex numbers $\z_n = (z_1,\dots,z_n)\in \C^n$ such that
 \[
 -\frac 12  \Re ( \bar z_n p_\a(i\z_{n-1}))=1,
 \]
 where $\a=(1,-1,\dots,-1)\in\A_{n-1}$.
The function $u=(w_{z_1;k_1},\dots, w_{z_n;k_n} )$ is in $V_{n-1}$, by Theorem~\ref{bugsbunny}, formulas \eqref{85} and \eqref{857}.
By formula \eqref{95} and by Lemma~\ref{lemma85}
\[
 | I^{{id}}(u) |= |c_{\bar \a}^{n-1}(\k_{n-1})|> \sum_{\s\in S_n^1,  \s\neq id}  | c_{\bar\a}^{n-1}( k_{\s _1},\dots, k_{\s_{n-1}} )|=
 \sum_{\s\in S_n^1,  \s\neq id}  | I^\s(u) |.
\]
For each $\tau \in S_n^1$, we define $u_\tau = (w_{z_{\bar 1};k_{\bar 1}},\dots, w_{z_{\bar n};k_{\bar n}} )$, where $\bar \ell =\tau^{-1}(\l) $.
As above, we have
\[
 | I^{\tau}(u_\tau ) | > \sum_{\s\in S_n^1,  \s\neq \tau }   | I^\s(u_\tau) |.
\]
This concludes the proof that $(I^\sigma(u_\tau))_{\s,\tau\in S_n^1}$ is strictly diagonally dominant.

 \end{proof}

 \section{Goh conditions of order $n$ in the corank 1 case}

\label{GOH}

Let $\Delta\subset TM$ be the distribution spanned point-wise by the vector fields $f_1,\dots, f_d$. For any $n\in\N$ and $q\in M$ we let
\[
\Delta_n(q) = \mathrm{span}_\R \big \{ [f_{\j_n},\dots,f_{\j_1}] (q) \mid \j_1,\dots, \j_n \in \{1,\dots,d\} \big\}\subset T_qM.
\]
 For a given $q\in M$, the annihilator of $\Delta_n$ is 
 \[
 \Delta_n^\perp(q) = \big\{ \lambda \in T_q^* M \mid \lambda (v)= 0\textrm{ for all $v\in \Delta_n(q)$}\big\}.
 \]

	A horizontal curve $\gamma \in AC ( I ; M)$ is {\em regular} ({\em singular} or {\em strictly singular}) if its  control  is regular (singular or strictly singular). The corank of $\gamma$ is the corank of its control.

Let $\gamma: I \to M$ be a horizontal curve with control $u$,   $\gamma(0) = q$ and $\gamma(1) = \q $. 
We denote by $P_{0}^t	$ the flow of the non-autonomous vector field $f_{u}= \sum_{i=1}^ d u_i f_i$. 
Then we have $\gamma(t) = P_0^t(q)$ for $t\in I $ and the differential $(P_t^{1})_*:
T_{\gamma(t)}M \to  T_{\q }M$ is given by 
\[
	(P_t^{1 })_*= 
	\mathrm{Ad}\left(\eexp\int_1^t f_{u(\tau)}d\tau\right).
\]
We refer the reader to \cite[Chapter 2]{AgrSac} for a detailed introduction to the formalism of chronological calculus.
We denote by $(P_t^{1})^*: T^*_{\q }M\to T_{\gamma (t)}^*M$ the  adjoint map and for every   $\lambda\in \IM(d_0G)^\perp$, the curve of covectors defined by 
\be\label{AGG}
\lambda(t)=(P_t^{1})^*\lambda\in T_{\gamma(t)}^*M,
\quad  t\in  I ,
\ee
is called the \emph{adjoint curve}  to $\gamma$ relative to $\lambda$. In the corank 1 case, this curve is unique  up to normalization of $\lambda \neq 0$.

\begin{theorem}
 \label{Gohn}
 Let $(M,\Delta,g)$ be a sub-Riemannian manifold,   $\gamma \in AC ( I ; M)$ be a horizontal curve with control $u\in X$, and   $n\in\N$ be an integer with $n\geq 3$.
 Assume that:
 \begin{itemize}
 \item[i)] { \color{black} $\dom(\mc{D}^n_u F)$ has finite codimension}, where  $F$ is the end-point map starting from $\gamma(0)$;
 
 \item[ii)] $\gamma$ is a strictly singular length-minimizing curve with corank $1$.
 \end{itemize}
 \noindent
 Then any  
 adjoint curve $\la \in AC( I ;  T^*M)$ satisfies 
 \begin{equation}
 \label{DELTA}
 \lambda(t) \in \Delta_n^\perp(\gamma(t)) \quad \textrm{for all  $t\in I $}.
 \end{equation} 
 \end{theorem}

\begin{proof}
 If $\gamma$ is length-minimizing then the extended end-point
map $F_J$ is not open at $u$, i.e., the extended variation map $G_J$ is not open at $0$.
By Theorem~\ref{endopen2} we consequently have $\mc{G}_{t_0}^n(v)= 0$ for all $t_0\in I $  and for all $v\in V_{n-1}$. In order to use 
Theorem~\ref{endopen2} we need both assumptions i) and ii).
The map $\mc{G}_{t_0}^n$ is introduced in \eqref{mom} and incorporates $\lambda$.
The strict singularity of $\gamma$ is used to translate the differential analysis from $G_J$ to $G$.

We  polarize the equation $\mc{G}_{t_0}^n(v)= 0$, as explained at the end of  Section~\ref{OMP}.
The polarization, denoted by $\T$, is introduced in \eqref{PO}. We   have $\T=0$ on \emph{linear} spaces contained in $V_{n-1}$.
We fix any selection function $\j:\{1,\dots,n\}\to \{1,\dots,d\}$ and we translate our claim \eqref{DELTA} into the new claim 
 \begin{equation}
  \label{goheq}
  \langle \la(t),[f_{\j_n},\.,f_{\j_1}](\g(t)) \rangle=0,\quad t\in  I .
 \end{equation} 
By   formula \eqref{mainpolfin2} for $\T$ proved in Theorem~\ref{lisk1},  if we choose $v_1,\dots,v_n \in X$ as in \eqref{vi} and such that the corresponding $u$ satisfies  $u\in V_{n-1}$, then
the equation $\T(v_1,\dots,v_n)=0$ reads
\begin{equation}
\label{PUF}
 \sum_{\s\in S_n^1} 
 \langle \lambda, [g^{t_0}_{\j\s } ]\rangle  I^{\sigma}(u) =0,\quad t_0\in I .
\end{equation}
We regard \eqref{PUF} as a linear equation in the unknowns  $\langle \lambda, [g^{t_0}_{\j\s } ]\rangle$ with coefficients 
$I^{\sigma}(u)$.

By Theorem~\ref{TRIX}, for any $\tau \in S_n^1$ there exists $u_\tau \in V_{n-1}$ such that the matrix $( I^{\sigma}(u_\tau))_{\s,\tau\in S_n^1}$
is invertible. From \eqref{PUF}, the definition \eqref{AGG}
of adjoint curve, and  \eqref{gufo} we deduce that for any $\s\in S_n^1$ and  $ t_0\in I $  
\begin{equation}
\label{PIF} 
\begin{split}
 0 = \langle \lambda, [g^{t_0}_{\j\s } ]\rangle    
 & =\langle \lambda, [g^{t_0}_{\j\s_ n },\dots,  g^{t_0}_{\j\s _1 }]\rangle  
 \\
 &=\langle \lambda, [(P^1_{t_0}  )_* f_{\j\s_ n } ,\dots,  (P^1_{t_0} )_*f_{\j\s _1 } ]\rangle  
\\
 &=\langle \lambda,(P^1_{t_0} )_* [f_{\j\s_ n },\dots,  f_{\j\s _1 }]\rangle  
 \\
 &
 =\langle (P^1_{t_0} )^*\lambda , [f_{\j\s _n },\dots,  f_{\j\s _1 }](\gamma(t_0)) \rangle 
 \\
 &
 =\langle  \lambda(t_0) , [f_{\j\s _n },\dots,  f_{\j\s _1 }](\gamma(t_0)) \rangle.
 \end{split}
\end{equation}
This identity with $\s=id$ is   \eqref{goheq}.

\end{proof}
%
%
%
%
%

 \begin{remark}
 \label{rk93}
 The inverse implication in Theorem \ref{Gohn} does not hold. Namely, a strictly singular curve satisfying assumption i) and \eqref{DELTA}  in Theorem \ref{Gohn} needs not be length-minimizing. A counterexample  is given in the next section.
\end{remark}

\section{An example of singular extremal} 

On the manifold $M=\R^3$, with coordinates  $x=(x_1,x_2,x_3)\in\R^3$,
we consider the rank $2$ distribution  $\De^n=\mr{span}\{f_1,f_2\}$,
where $f_1$ and $f_2$ are the vector-fields in \eqref{eq:vfieldsintro}. 
The vector-field $f_2$ depends on the parameter 
 $n\in\N$. We fix on $\De^n$ 
 the metric $g$ making $f_1$ and $f_2$ orthonormal.

In this section, we study the (local) length-minimality in $(\R^3,\De^n,g)$ of the curve
$  \g:I= [0,1]  \to\R^3$
\begin{equation}
 \label{eq:gamma} \gamma(t)=  (0,t,0), \quad t\in I. 
 \end{equation}
The curve $\g$ is in fact defined for all $t\in\R$.

Our results rely upon  the analysis of the variation map $G$  introduced in Section 5,  $G(v)=G_{\q}^u(v)=F_0(u+v)$, where $F=F_0$ is the end-point map with starting point $q=0$,  $\q =\g(1)=(0,1,0)$ is the end-point, and $u=(0,1)\in L^1(I;\R^2)$ is the control of $\g$.
The extended maps $F_J,G_J$ are   defined as in Sections 5-6.

The minimality properties of $\g$ are described in the following theorem.

\begin{theorem} 
\label{teoes}
For $n\in\N$, let us consider the sub-Riemannian manifold $(\R^3,\De^n,g)$ and the curve $\g$ in \eqref{eq:gamma}.
\begin{itemize}
 \item[i)] For any $\m\geq2$, $\g$ is the unique  strictly singular extremal in $(\R^3,\Delta^n)$
passing through the origin,
up to reparameterization. 

\item[ii)] If $\m\geq 2$ is even,   $\g$ is locally length-minimizing in $(\R^3,\De^n,g)$.
\item[iii)] If $\m\geq 3$ is odd,   $\g$ is not  length-minimizing in $(\R^3,\De^n,g)$, not even locally.
\end{itemize}
\end{theorem}

Above, \enquote{locally length-minimizing} means that short {\color{black} enough}   sub-arcs of $\g$ are length-minimizing for fixed end-points.
Claims i) and ii) are well-known. In particular, claim ii) can be proved with a straightforward adaptation of Liu-Sussmann's argument for $n=2$ in \cite{LS95}.  For $\m=3$, claim iii) is proved in  \cite{BMP20} and here we prove the general case.

We compute  the $n$th intrinsic differential of $G$ and we show that, for odd $n$, it satisfies  the hypotheses of Corollary \ref{REGONO} part ii). This implies that  the extended variation map $G_J$ is open and, as a consequence, the non-minimality of $\g$. We will also show that the lower intrinsic differentials vanish, $\mc D_0^h G=0$ for $h<n$.

We denote by  $\g^x$ the horizontal curve with control $u=(0,1)$ and $\g^x(1)=x$, so that   $\g^{\q }=\g$. By the  formulas in \eqref{eq:vfieldsintro}  for the vector fields $f_1$ and $f_2$, we find
 \[
  \g_1^x(t)=x_1, \quad \g_2^x(t)=(t-1)(1-x_1)+x_2, \quad  \g_3^x(t)=(t-1)x_1^\m+x_3.
 \]
 The \enquote{optimal flow} associated with $\gamma$ is 
  the $1$-parameter family of diffeomorphisms   $P_1^t\in C^\infty(\R^3;\R^3)$, $t\in\R$,
  defined by $ P_1^t(x) =\g^x(t)$.
  For fixed $x\in \R^3$, 
  the inverse of the  differential of $P_1^t$  is the map $P_1^t(x)_*^\1=P^1_t(x)_*:T_{\g^x(t)}\R^3\to T_{x }\R^3$  
   \[
 P^1_t(x)_*=
  \begin{pmatrix}
   1 & 0 & 0 \\
   t-1 & 1 & 0 \\
   -\m(t-1)x_1^{\m-1} & 0 & 1
  \end{pmatrix}.
 \]
 
 As explained in Section \ref{EXPA}, see formula  \eqref{PLAT},
 the differential of $G$ at 0 is 
 \[
   d_0G(v)=\int_0^1 g^t _{v(t)}(\q)dt,  
 \quad  v=(v^1,v^2)\in L^2(I;\R^2).
 \]
Above, we set $ g^t_{v(t)}=v^1(t)g_1^t+v^2(t)g_2^t$, where  the vector fields $g_1^t $ and $g_2^t $ are
 \begin{align*}
  g_1^t &=P_t^1(x)_*f_1= \frac{\d}{\d x_1} + (t-1)\frac{\d}{\d x_2} - \m(t-1)x_1^{\m-1}\frac{\d}{\d x_3},\\
  g_2^t &=P_t^1(x)_*f_2=f_2=(1-x_1)\frac{\partial}{\partial x_2}+x_1^\m\frac{\partial}{\partial x_3}.
 \end{align*}
 So the differential is given by the formula
 \be
  \label{diffG}
  d_0G(v)=
  \begin{pmatrix}
  \displaystyle  \int_0^1v^1(t)dt\\
    \displaystyle 
   \int_0^1\big\{(t-1)v^1(t)+v^2(t)\big\}dt\\
   0
  \end{pmatrix}.
 \ee
 We deduce that a generator for $\IM(d_0G)^\perp$ is the covector $\la=(0,0,1)$, and that  $v\in\ker(d_0G)$ if and only if
 \begin{equation}
  \label{stanelker}
  \begin{split}
   \int_0^1 v^1(t)dt=0 \qquad \mr{and} \qquad
   \int_0^1 (tv^1(t)+v^2(t))dt=0.
  \end{split}
 \end{equation}
 
In  the computation of the differentials $\D_0^hG$, $h\geq 2$, we need the following lemma. For $y\in L^2(I;\R)$ and $n\geq 2$,
we let  
\begin{equation}
 \label{Iy}
 \Ga_y^n=\int_{\Sigma_n} y(t_1)\.y(t_n)(t_2-t_1)d\L^{n}.
\end{equation}

\begin{lemma}
 \label{lemma102}
 For  $n\geq 2$ and  $y\in L^2(I;\R)$   such that $\displaystyle\int_0^1 y(t)dt=0$ we have 
 \begin{equation}
 \label{segno}
 \Ga_y^n=\frac{1}{n!}\int_{0}^1\left(\int_{t}^1y(\tau)d\tau\right)^ndt.
\end{equation}
\end{lemma}

\begin{proof}
  We first observe that, integrating by parts, we have
 \begin{align*}
  \int_{t_2}^1y(t_1)(t_2-t_1)dt_1&=t_2\int_{t_2}^1y(t_1)dt_1-\int_{t_2}^1t_1y(t_1)dt_1\\
  &=t_2\int_{t_2}^1y(t_1)dt_1-\left[s\int_{s}^1y(t_1)dt_1\right]_{s=t_2}+\int_{t_2}^1\int_{s}^1y(t_1)dt_1ds\\
  &=\int_{t_2}^1\int_{s}^1y(t_1)dt_1ds.
 \end{align*}
 Applying this identity to $\Ga_y^n$ and integrating by parts again, we get
 \begin{align*}
  \Ga_y^n&=\int_{0}^1y(t_n)\int_{t_n}^1\dots\int_{t_2}^1\int_s^1y(t_1)dt_1dsdt_2\.dt_n\\
  &=\int_{0}^1y(t_n)dt_n \int_{0}^1y(t_{n-1})\int_{t_n}^1\dots\int_{t_2}^1\int_s^1y(t_1)dt_1dsdt_2\.dt_{n-1}\\
  &\quad\;-\int_{0}^1\left(\int_{t_n}^1y(\tau)d\tau\right) y(t_{n})\int_{t_n}^1y(t_{n-2})\dots\int_{t_2}^1\int_s^1y(t_1)dt_1dsdt_2\.dt_{n-2}dt_n\\
  &=\frac{1}{2}\int_{0}^1\frac{d}{d t_n}\left(\int_{t_n}^1y(\tau)d\tau\right)^2\int_{t_n}^1y(t_{n-2})\dots\int_{t_2}^1\int_s^1y(t_1)dt_1dsdt_2\.dt_{n-2}dt_n.
 \end{align*}
 In the last identity, we used our assumption $\displaystyle \int_0^1 y(t)dt=0$.
Now our claim follows by iterating this   integration by parts argument.
\end{proof}

\begin{theorem}
 \label{teo103}
 Let $n\in\N$. The variation map $G$ in $(\R^3,\De^n)$ satisfies:
 \begin{itemize}
  \item[i)]  $\D_0^hG=0$, for $ h<\m$;
  \item[ii)] for any $v=(v_1,\.,v_{n-1})\in\dom(\D_0^nG)$,
  \begin{equation}
  \label{eqDn} \D_0^nG(v)= \int_0^1\left(\int_t^1 v_1^1(\tau)d\tau\right)^n dt,
  \end{equation}
  where $v_1^1$ is the first coordinate of $v_1$.
 \end{itemize}
\end{theorem}

\begin{proof}
The Lie brackets of the vector fields $g_1^t$ and $g_2^t=f_2$ are, at different times,
 \begin{align*}
  [g_{1}^t ,g_{1}^{s}]&=\m(\m-1)(t-s)x_1^{\m-2}\frac{\d}{\d x_3},\\
  [g_{1}^t ,g_{2}^{s}]&=-\ddue+\m x_1^{\m-1}\dtre.
 \end{align*}
 Notice that the bracket in the latter line is time-independent.
  Then, for $3\leq h\leq\m$ and $i_1,\.,i_h\in\{1,2\}$, the   iterated brackets of length $h$ are 
 \[
  [g^{t_h}_{i_h},\.,g^{t_1}_{i_1}]=
  \begin{cases}
   \m\.(\m-h+1)(t_2-t_1)x_1^{\m-h}\dtre, &\text{ if } i_1=\.=i_h=1,\\
   \m\.(\m-h+2)(t_2-t_1)x_1^{\m-h+1}\dtre, &\text{ if } i_2=\.=i_h=1, \text{ and } i_1=2,\\
   0, &\text{ otherwise}.
  \end{cases}  
 \]

For $ h<\m$, the coefficient  of $\partial /\partial x_3$  in the formulas above vanishes  at the point $\q =(0,1,0)$ and thus
the projection of these brackets along the covector  $\la=(0,0,1)$ vanishes, for any $2\leq h<n$,
\[
 \langle \la,[g^{t_h}_{i_h},\.,g^{t_1}_{i_1}](\q) \rangle =0.
\]
Using formulas \eqref{ciao} and \eqref{PLAT}, we deduce that   for any   $(v_1,\.,v_{h-1})\in\dom(\D_0^hG)$ we have
\[
 \D_0^hG(v_1,\.,v_{h-1})=\langle \la,D_0^hG(v_1,\.,v_{h-1},*)(\q) \rangle=0, 
 \]
proving claim i).

For $h=\m$,   
the coefficient  of $\partial /\partial x_3$   vanishes  at  $\q $ except  for the case $i_1=\.=i_n=1$, that is
\begin{align*}
  \langle \la,[g^{t_n}_{i_n},\.,g^{t_1}_{i_1}](\q) \rangle &=0, \quad \text{if } i_j\neq1 \text{ for some } j,\\
  \langle \la,[g^{t_n}_1,\.,g^{t_1}_1](\q) \rangle &=\m!(t_2-t_1).
\end{align*}
Then, for any $v=(v_1,\.,v_{n-1})\in\dom(\D_0^nG)$ we have
\begin{align*}
 \D_0^\m G(v)&=\langle \la,D_0^nG(v_1,\.,v_{\m-1},*) \rangle=\m!\int_{\Sigma_n} v_1^1(t_1)\.v_1^1(t_n)(t_2-t_1)d\L^{n}\\
 &= \int_0^1\left(\int_t^1 v_1^1(\tau)d\tau\right)^n dt.
\end{align*}
In  the last identity we used Lemma \ref{lemma102}. This proves claim ii).
\end{proof}

Before proving claim iii) of Theorem \ref{teoes}, we recall that $\ker(d_0G_J)=\ker(d_0G)\cap\ker(d_uJ)$, where $J$ is the energy functional (see Section 6). In particular, for any $v\in L^2(I;\R^2)$ we have
\be
 \label{diffJ}
d_uJ(v)=
 \int_0^1\big(u^1(t)v^1(t)dt+
   u^2(t)v^2(t)\big)dt
 =\int_0^1v^2(t)dt.
\ee

\begin{proof}[Proof of Theorem \ref{teoes} - claim iii)]
Let $\m\in\N$ be an odd integer. We claim that there exists  $v=(v_1,\.,v_{n-1})\in\dom(\D_0^nG_J)   \subset \dom(\D_0^nG)$ such that $\D_0^nG(v)\neq0$. The inclusion of domains is ensured by \eqref{621}.
By Theorem \ref{teo103} we have $\D_0^hG=0$ for any $h<\m$. Then from \eqref{62} it follows that also the extended map
satisfies   $\D_0^nG_J(v)\neq0$  and 
 $\D_0^hG_J=0$
 for $h<\m$.
 
 By  Proposition \ref{REGON}
  the differential $\D_0^nG_J$ is regular; here, we are using the fact that $n$ is odd. 
By Theorem \ref{thm:Gopen},   $G_J$ is open at zero and thus  $F_J$ is open at $u$.
This implies that $\g$ is not length-minimizing.

So, the proof of our claim reduces to find a function $v_1\in\ker(d_0G_J)$ such that 
\be
 \label{108}
 \int_0^1\left(\int_t^1 v_1^1(\tau)d\tau\right)^n dt\neq0.
\ee
If such a control $v_1$ exists, then by Proposition \ref{prop:corrrections} there also exist  $v_2,\.,v_{n-1}\in L^2(I;\R^2)$ such that $v=(v_1,\.,v_{n-1})\in\dom(\D_0^nG_J)$ and by \eqref{eqDn} it follows $\D_0^nG(v)\neq0$.

The condition $v_1\in\ker(d_0G_J)$ is equivalent to $d_0G(v_1)=0$ and $d_uJ(v_1)=0$. By \eqref{diffG} and \eqref{diffJ} this means
\be
 \label{109}
 \int_0^1 v_1^1(t)dt=\int_0^1 tv_1^1(t)dt=\int_0^1v_1^2(t)dt=0.
\ee
We choose any funtion $v_1^2$ with vanishing mean.
Also choosing $v_1^1(t)=\chia(t)-5\chib(t)+3\chic(t)$, all the conditions in \eqref{109} are satisfied. Moreover, we have
\begin{align*}
 \int_t^1 v_1^1(\tau)d\tau=-t\chia(t) + (5t-3)\chib(t) - 3(t-1)\chic(t),
\end{align*}
and then, after a short computation,
\begin{align*}
 \int_0^1\left(\int_t^1 v_1^1(\tau)d\tau\right)^n dt=\frac{6}{5(n+1)4^{n}}\big(3^{n-1} - 2^{n-1}\big).
\end{align*}
The last quantity is different from 0 for any odd $n\geq3$, completing the proof.

 \end{proof}

 \begin{remark} \label{R104}
  We briefly comment on claim ii) of Theorem \ref{teoes}. By formula \eqref{eqDn},  when $n$ is even  we have $\D_0^nG(v)\geq 0$  for any $v\in\dom(\D_0^nG)$. {\color{black}
  So condition i) in Proposition \ref{REGON} is not  satisfied
  and we cannot use  our open mapping theorems, in particular Corollary \ref{REGONO}.}   
  Though not sufficient to prove the local minimality of $\g$, this is consistent with claim ii) of Theorem \ref{teoes}.
 \end{remark}

 \begin{remark}
  Claim iii) of Theorem \ref{teoes} answers the question raised in Remark \ref{rk93}. By Theorem \ref{teo103},  the curve $\g$ satisfies assumption i) of Theorem \ref{Gohn} for any $n\in\N$.  
  On the other hand, the non-vanishing Lie brackets of the vector fields $f_1$ and $f_2$ are  
  \begin{align*}
   [f_1,f_2]&=-\ddue+nx^{n-1}\dtre,\\
   [\underbrace{f_1,\.,f_1}_\text{$(h-1)$-times},f_2]&=n(n-1)\.(n-h+2)x_1^{n-h+1}\dtre.
  \end{align*}
  Then, for any $h\leq n$ we have
  \be
   \label{105}
   \langle \la,[\underbrace{f_1,\.,f_1}_\text{$(h-1)$-times},f_2](\g(t)) \rangle=0, \quad \text{for any } t\in I.
  \ee
 For $2\leq h<n$, this  is consistent with the vanishing of the $h$-th differential.
 When $h=n$, identity \eqref{105} is the Goh condition of order $n$  in \eqref{goheq}.
  
  Thus, if $n$ is odd the curve $\g$ satisfies both assumption i) of Theorem \ref{Gohn} and condition \eqref{goheq}.
  However, it is a strictly singular curve which is not length-minimizing. 
  
 \end{remark}

 \bibliographystyle{plain}

\bibliography{BiblioEndPoint}
   
  \end{document}